\documentclass[10pt,leqno,a4paper]{article}
\usepackage{amsfonts,amsmath,amssymb}
\usepackage{amsthm}
\usepackage{graphics,url,color,epsfig}
\usepackage[hidelinks]{hyperref}
\usepackage{subfigure}
\usepackage{stmaryrd}
\usepackage{enumitem}

\usepackage{enumerate}
\usepackage{algorithm}
\usepackage{algorithmicx}
\usepackage{multirow}

\usepackage{geometry}
\newcommand{\mfrac}[1][2]{\frac{1}{2}}

\newtheorem{theorem}{Theorem}[section]
\newtheorem{lemma}{Lemma}[section]

\newtheorem{example}{Example}[section]
\newtheorem{remark}{Remark}[section]
 \allowdisplaybreaks

\title{Fast time-stepping discontinuous Galerkin method for the subdiffusion equation
\thanks{This work has been supported by the Project of the National Key R\&D Program (Grants No. 2021YFA1000202), the National Natural Science Foundation of China (Grants Nos. 12120101001, 12131005, 12001326, 12171283), the Natural Science Foundation of Shandong Province (Grant Nos. 2022HWYQ-045, ZR2021ZD03, ZR2020QA032).}
}

\author{Hui Zhang${}^\dag$, Fanhai Zeng${}^\dag$, Xiaoyun Jiang\thanks{School of Mathematics, Shandong University, Jinan 250100, China
(zhangh@sdu.edu.cn, fanhai\_zeng@sdu.edu.cn, wqjxyf@sdu.edu.cn).}
, Zhimin Zhang\thanks{Wayne State University Detroit, MI 48202, USA (ag7761@wayne.edu).}
}

\begin{document}
\maketitle
\textbf{Abstract.} The nonlocality of the fractional operator causes numerical difficulties for long time computation of the time-fractional evolution equations. This paper develops a high-order fast time-stepping discontinuous Galerkin finite element method for the time-fractional diffusion equations, which saves storage and computational time. The optimal error estimate $O(N^{-p-1} + h^{m+1} + \varepsilon N^{r\alpha})$ of the current time-stepping discontinuous Galerkin method is rigorous proved, where $N$ denotes the number of time intervals, $p$ is the degree of polynomial approximation on each time subinterval, $h$ is the maximum space step, $r\ge1$, $m$ is the order of finite element space, and $\varepsilon>0$ can be arbitrarily small. Numerical simulations verify the theoretical analysis.

\textbf{Key words.} fast time-stepping discontinuous Galerkin method, optimal convergence, subdiffusion, nonlocality.

 \textbf{MSC codes.} 26A33, 65M06, 65M12, 65M15, 35R11

%





%
%
%


\section{Introduction}\label{intro}
The aim of this paper is to develop the fast time-stepping discontinuous Galerkin (DG)
finite element method (FEM) for the following subdiffusion equation
\begin{equation}\label{fpde}\left\{\begin{aligned}
&{}_{0}^{C} D_{t}^{\alpha}u(x,t) + \mathcal{L} u(x,t)=f(x,t), && (x,t)\in \Omega\times (0,T],T>0,\\
&u(x,0)=u_0(x),&& x \in \bar{\Omega},\\
&u(x,t)=0, &&(x,t)\in \partial\Omega\times [0,T],
\end{aligned}\right.\end{equation}
where $x = (x_1, x_2, ..., x_d)\in \Omega\subset \mathbb{R}^d$ is a bounded Lipschitz domain, $d \in \{1, 2, 3\}$, $\mathcal{L}$ is a linear second-order elliptic operator, and
the Caputo fractional derivative operator $_{a}^{C}D_{t}^{\alpha}$  is defined by
\begin{equation}\begin{aligned}\label{e1.133}
{}_{a}^{C}D_{t}^{\alpha}u(x,t)=\frac{1}{\Gamma(1-\alpha)}
\int_{a}^{t}\frac{\partial u(x,s)}{\partial s}(t-s)^{-\alpha}\mathrm{d}s,\quad 0<\alpha<1,\,\,a\in \mathbb{R}.
\end{aligned}\end{equation} For simplicity of analysis, we consider $\mathcal{L}=-\Delta$ in the paper.
The subdiffusion equations,
which can model the sublinear growth of mean squared particle displacement in transport processes, have attracted much interests of physicists, engineers and applied mathematicians in developing highly accurate and efficient computational methods
and numerical analysis.

The regularity of the solution of  \eqref{fpde} has been fully understood in literature \cite{JinZhou2023,LiMa2022,Stynes21}. In this paper, we  assume that the solution $u(t)=u(\cdot,t)$ of \eqref{fpde} satisfies the following finite  regularity:
\begin{equation}\label{e10004-1}\begin{aligned}
&\|u^{(\ell)}(t)\|_{H^{q}(\Omega)}\le C(1+ t^{\sigma-\ell}),\quad t>0,
\end{aligned}\end{equation}
where $C>0$, $\sigma>0$, $0\le \ell \le p+1$, $p$ and $q$ are  nonnegative integers depending the smoothness of the initial data and the source term $f(x,t)$, and $u^{(\ell)}(t)$ denotes the $\ell$-th order time derivative of $u(t)=u(\cdot,t)$; see  \cite{LiMa2022}, in which $\sigma=\alpha$ can be obtained. From \eqref{e10004-1}, one will find that there exists  the initial singularity of the solution
to  \eqref{fpde}. We employ  the adaptive mesh  to deal with
the numerical difficulty of the singularity, see, e.g., \cite{stynes2017error}.

The non-locality of the fractional derivative operator \eqref{e1.133} causes a lot of numerical difficulty for solving \eqref{fpde}, especially for long time computation and high dimension space models. Generally speaking, the approximation of the  fractional derivative ${}_{0}^{C} D_{t}^{\alpha}u(t_n)$ takes the form of (see, e.g., \cite{Alikhanov2015,SunWu06,LinXu07})
$$\sum_{j=0}^nw_{n,j}u(t_j),\quad 1 \le n\leq N,$$ where $w_{n,j}$ depends on specific time discretization methods, the fractional order $\alpha$, and the temporal mesh.

There have been a lot of numerical methods for solving \eqref{fpde}, readers can refer to  papers \cite{jin2017correction,karaa2017finite,LinXu07,Stynes21,zhu2019fast} and the books of fractional calculus \cite{Diethelm2010,JinZhou2023,li2015numerical,SunGao2020}. For the fast memory-saving time-stepping methods, readers can refer to \cite{baffet2017kernel,banjai2019efficient,chen2019accurate,guo2019efficient,jiang2017fast,lopez2008adaptive,ZengTB2018}, where the fast methods are designed to directly discretize the fractional integral or derivatives. Here, we extend the ideas of the fast methods in the aforementioned papers  to calculate the integral arising from the time-stepping DG methods. Time-stepping DG schemes have been developed to solve time-fractional diffusion and diffusion wave equations  as $u_t(t) + \partial_t^\beta(\mathcal{L}u)(t) =g(x,t),-1<\beta <1$, readers can refer to \cite{li2020analysis,li2020numerical,mustapha2015time,mustapha2012uniform} for more details.

To the best of our knowledge,
the only work of the time-stepping DG method for \eqref{fpde} was developed by
Mustapha and his collaborators in \cite{MusAFN16}, where
they proposed a linear time-stepping DG method (see \eqref{F-1} with $p=1$).
Under the regularity condition
$\|u'(t)\|_{H^2(\Omega)}+t\|u''(t)\|_{H^1(\Omega)}\le Ct^{\sigma-1},t>0,C\ge 0$, and $\sigma>\alpha/2$,
Mustapha et al.  proved the following error bound  \cite[Theorem 3]{MusAFN16}
\begin{equation}\label{eq-1}
\int_0^{t_n} \|U_h(t)-u(t)\|^2\textrm{d}t
\le C(N^{-(4-\alpha)} + h^{4}),\quad  r \ge \frac{4-\alpha}{2\sigma-\alpha},
\end{equation}
where $U_h(t)$ is the DG solution.
In this paper, we obtain the optimal error bound
\begin{equation}\label{eq-2}
\int_0^{t_n}\|U_h(t)-u(t)\|^2\textrm{d}t
\le C (N^{-2(p+1)} + h^{2m+2}),   \quad  r \ge \frac{2p+2-\alpha}{1+2\sigma-\alpha}
\end{equation}
for the high-order DG method for \eqref{fpde} under the condition \eqref{e10004-1}.
The case $p=1$ reduces to the linear DG method in \cite{MusAFN16}.
In our proof, we do not need the condition $\sigma>\alpha/2$.
The new error bound \eqref{eq-2} is optimal,
which  improves \eqref{eq-1}.
The new ingredient of our proof is that we introduce a new orthogonal projector
that helps to prove the optimal error bound in the time discretisation.

In order to reduce the storage and computational cost originated from the nonlocality of the time-fractional operator,
we develop  high-order memory-saving fast time-stepping DG method for \eqref{fpde}.
The following error bound is proved
\begin{equation}\label{eq-3}
\int_0^{t_n}\|U(t)-{}_FU(t)\|^2\textrm{d}t
\le C  (\varepsilon n^{r\alpha})^2,
\end{equation}
where ${}_FU$ is the solution of the fast time-stepping DG method (see \eqref{FDG-1}) and $\varepsilon>0$ is the error from the Gauss quadrature that can be made arbitrarily small. The  bound \eqref{eq-3} shows that the error of the fast DG method is independent
of the sizes  of the temporal  discretization error (Theorem \ref{thm3}).
To the best of  authors' knowledge, this is the first work  on the fast time-stepping DG scheme for solving the time-fractional
evolution equations as \eqref{fpde}. We also obtain the identical bound as \eqref{eq-3} for the fully discrete DG scheme, see Theorem \ref{thm4}.

This paper is organized as follows. Section \ref{sec2} presents the detailed error analysis of time-stepping DG methods for \eqref{fpde}.
Section \ref{sec3} develops the fast time-stepping DG methods and displays the detailed convergence analysis. Numerical simulations are given in Section \ref{sec4} to show the effectiveness of the present fast time-stepping DG schemes. The conclusion is given in the last section.

\section{Time-stepping DG methods}\label{sec2}
To describe the DG method, we introduce time grid points $0=t_0<t_1<\cdots<t_N=T$
and the half-open subintervals $I_n=(t_{n-1}, t_n]$ with the length $\tau_n = t_n -t_{n-1}$ for $1\le n \le N$.
We have $(0,t_n]=\cup_{k=1}^nI_k =J_n$.
In this paper, we use the following time grid points
\begin{equation}\begin{aligned}\label{TE-22}
t_n=(n/N)^r T\ \ \mathrm{for}\ \ 0\leq n \leq N,\quad r\ge 1.
\end{aligned}\end{equation}

In order to set up the   DG method, we let $\mathcal{P}_{k}(S)$ be the space of polynomials with degree
no greater than $k$ in the time variable $t$, with coefficients in the space $S$.
We introduce the trial space
\begin{equation*}
\mathcal{W}=\{v\in L^2((0,T),L^2(\Omega)): v|_{I_n}\in \mathcal{P}_p(L^2(\Omega)),\ 1\leq n \leq N\}.
\end{equation*}
For a function $v\in\mathcal{W}$, the left-hand limit and right-hand limit at $t_{n}$  can be respectively denoted by
\begin{eqnarray*}
v_{-}^{n}:= v(t_{n})=v(t_{n}^{-}) \quad\text{and}\quad v_{+}^{n}:=v(t_{n}^{+}).
\end{eqnarray*}
Throughout the paper, $\langle\cdot,\cdot\rangle$
denotes the $L^2$ inner product associated with the norm $\|\cdot\|=\|\cdot\|_{L^2(\Omega)}$,
and $\|\cdot\|_{H^k(\Omega)}$ denotes the norm on the Soblev space $H^k(\Omega)$.

For $a,\beta\in \mathbb{R}$, we define the left fractional operator ${}_a
\partial_t^{-\beta}$ as
\begin{equation}\begin{aligned}\label{RL}
{}_a\partial_t^{-\beta}u(t) = (\omega_{\beta}*u)(t)  = \int_a^t \omega_{\beta}(t-s)u(s)\textrm{d}s,\qquad
\omega_{\beta}(t)=\frac{t^{\beta-1}}{\Gamma(\beta)}.
\end{aligned}\end{equation}
For $\beta<0$,  \eqref{RL} is interpolated in terms of the principal value,
which is equivalent to the Riemann--Liouville (RL) fractional derivative operator of order $-\beta$ \cite{SamKilM93}.

We denote $\partial_t^{-\beta}={}_0\partial_t^{-\beta}$ for notational simplicity.
The following property is used in the formulation and computation of the (fast) time-stepping DG method
\begin{equation}\label{RL-C}
{}^C_aD_t^\alpha u(t)={}_a\partial_t^{\alpha} u(t)-\omega_{1-\alpha}(t-a)u(a).
\end{equation}

For notational simplicity, we introduce the notations $\lesssim $ and $ \gtrsim$.  Let  $A,B$ be the real numbers.
$A \lesssim B$ (or $A \gtrsim B$) means there exists a positive
constant $C$ independent of $\varepsilon$, the sizes of the time and/or space grids such
that $A \le CB$ (or $A \ge CB$).

\subsection{Semi-discrete time-stepping DG scheme}
According to \cite[Eq. (7)]{MusAFN16},
the semi-discrete time-stepping DG method for \eqref{fpde} is defined as: Given $U(t)$ for $ t \in [0,t_{n-1}]$, $n\ge 1$,
find $U\in\mathcal{P}_p(H_0^1(\Omega))$ on the next time subinterval $I_n$, such that
\begin{equation}\label{DG-1}
\int_{I_n} \left[\langle\partial_t^\alpha U,X \rangle+ \langle\nabla U, \nabla X \rangle\right]\textrm{d}t
=\int_{I_n} \langle f+\omega_{1-\alpha}(t)u_0,X \rangle\textrm{d}t,\,\, \forall X \in \mathcal{P}_p(H_0^1(\Omega)).
\end{equation}

The convergence of \eqref{DG-1} is presented in the following theorem.
\begin{theorem}
\label{thm1}
Let $u(t)\in H^1_0(\Omega)$ be the solution of \eqref{fpde} satisfying the regularity property \eqref{e10004-1} with $q=1$, $u_0,f\in L^2(\Omega)$. Let $U$ be the semi-discrete DG solution defined by \eqref{DG-1}.
Then
\begin{equation} \label{eq:error-a}
\int_0^{t_n}  \|U(\cdot,t)-u(\cdot,t)\|^2
{\lesssim \mathcal{E}(n,N,\sigma,\alpha,r,p),}
\end{equation}
where
\begin{equation}\label{En-1}
\mathcal{E}(n,N,\sigma,\alpha,r,p)=:\left\{\begin{aligned}
&N^{-r(1+2\sigma-\alpha)-\alpha},&&1\le r< \frac{2p+2-\alpha}{1+2\sigma-\alpha},\\
&N^{-2p-2}(1+\ln(n)),&&r= \frac{2p+2-\alpha}{1+2\sigma-\alpha},\\
&N^{-2p-2}t_n^{1+2\sigma-\alpha-\frac{2p+2-\alpha}{r}},&&r> \frac{2p+2-\alpha}{1+2\sigma-\alpha}.
\end{aligned}\right.
\end{equation}
\end{theorem}

\subsection{Fully discrete  time-stepping DG scheme}
Let $S_h\subset H^1_0(\Omega)$ denote the space of continuous,
piecewise polynomials of degree no greater than $m$
with respect to a quasi-uniform partition of $\Omega$ into conforming triangular finite elements,
with maximum diameter $h$. The DG FEM space is defined as
$$\mathcal{W}_h=\{w\in L^2((0,T), S_h): w|_{I_{n}}\in \mathcal{P}_{p}(S_{h}) \ \textrm{for}\ 1\leq n\leq N\}.$$

The DG FEM  for \eqref{fpde} reads as:
Given $U_h(t)$ for $0\leq t \leq t_{n-1}$, find $U_h\in \mathcal{P}_{p}(S_{h})$ on $I_n$, such that
\begin{equation}\label{F-1}
\int_{I_n}[\langle \partial_t^\alpha U_h(t), X\rangle+ \langle \nabla U_h, \nabla X  \rangle]\textrm{d}t=
\int_{I_n} \langle f+\omega_{1-\alpha}(t) u_0,X \rangle\textrm{d}t,\forall X \in \mathcal{W}_h.
\end{equation}

We have the following theorem.
\begin{theorem}\label{thm-2}
Let $u(t)\in H_0^1(\Omega)$ be the solution of  \eqref{fpde}  satisfying   \eqref{e10004-1} with $q=m+1$, $u_0,f \in L^2(\Omega)$.
Let $U_h$ be the solution defined by \eqref{F-1}. Then, we have
\begin{equation}\label{error-1}
\int_0^{t_n}\|U_h(t)-u(t)\|^2\textrm{d}t \lesssim \mathcal{E}(n,N,\sigma,\alpha,r,p)+h^{2m+2}.
\end{equation}
\end{theorem}


\subsection{Convergence analysis}
Define $A_{\alpha}^n(\cdot,\cdot)$ as
\begin{equation} \label{Analf}
A_{\alpha}^n(u,v)=\int_{0}^{t_n} \langle \partial_t^\alpha u(t),v(t) \rangle\textrm{d}t.
\end{equation}

Define the projector $R_t^{\alpha}:J_L^{\alpha}(J_n)\rightarrow
\{v: v\in P_p(I_k),1\le k \le n\}$ as
\begin{equation} \label{Rtalf}
\int_{0}^{t_n} \partial_t^\alpha (R_t^{\alpha}u-u)(t)X(t) \textrm{d}t=0,
 \quad \forall X\in \{v: v\in P_p(I_k),1\le k \le n\},
\end{equation}where $J_L^{\alpha}(J_n)$ denotes the left factional derivative space defined in Appendix \ref{appx-a},
and $P_p(I_k)$ is the polynomial of order $p$ defined on $I_k$.
The above projection will be used in the convergence analysis of the DG schemes developed in
the previous subsections.

Some technical lemmas are given before the proofs of Theorems \ref{thm1} and \ref{thm-2}.

\subsubsection{Lemmas}
We present  some technical  lemmas in this section.

\begin{lemma}[{\cite[Theorem 4.8]{NetoJunnior20}}]\label{cor-4a}
Let $q>1/(1-\alpha)$ and $q\ge 2$. If $v\in L^q(0,T;\Omega)$, $\partial_t^{\alpha-1}v \in W^{1,2}(0,T;\Omega)$
and $\partial_t^{\alpha-1}v^2 \in W^{1,1}(0,T;\Omega)$, then
$$\partial_t^{\alpha}(v^2)(t)\le 2v(t)\partial_t^{\alpha}v(t).$$
\end{lemma}

\begin{remark}\label{lem-4a}
From Lemma \ref{cor-4a}, if $v\in \mathcal{W}_h$ or $\mathcal{W}$, then
$\partial_t^{\alpha}(v^2)(t)\le 2v(t)\partial_t^{\alpha}v(t)$.
\end{remark}


\begin{lemma}\label{lem-4}
Let $v\in \mathcal{W}$. Then
\begin{eqnarray}
A_{\alpha}^n(v,v)\ge \frac{1}{2} \int_{0}^{t_n}\frac{(t_n-t)^{-\alpha}}{\Gamma(1-\alpha)} \|v(t)\|^2\textrm{d}t
\ge\frac{t_n^{-\alpha}}{2\Gamma(1-\alpha)}\int_{0}^{t_n} \|v(t)\|^2\textrm{d}t.\label{FDG-err-11c2}
 \end{eqnarray}
\end{lemma}
\begin{proof}
For $v\in \mathcal{W}$, by Remark \ref{lem-4a}, we have
\begin{equation*}\begin{aligned}
A_{\alpha}^n(v,v)=&\int_{0}^{t_n}\langle  \partial_t^{\alpha} v(t),v(t) \rangle \textrm{d}t
\ge \frac{1}{2}\int_{0}^{t_n}\partial_t^{\alpha} \|v(t)\|^2 \textrm{d}t \\
=& \int_{0}^{t_n}\frac{(t_n-t)^{-\alpha}}{2\Gamma(1-\alpha)} \|v(t)\|^2 \textrm{d}t
\ge \frac{t_n^{-\alpha}}{2\Gamma(1-\alpha)}\int_{0}^{t_n} \|v(t)\|^2 \textrm{d}t.
\end{aligned}\end{equation*}
The proof is complete.
\end{proof}


The following Lemma \ref{lem-8} displays the convergence of the projectors $R_t^{\alpha}$, which plays an important role in proving the optimal convergence of the DG method. The proof of Lemma \ref{lem-8} is presented Appendix \ref{appx-a}.
\begin{lemma}\label{lem-8}
Let $v(t)$ satisfy
$|v^{(\ell)}(t)| \lesssim 1+ t^{\sigma-\ell}$ for $ t>0,\sigma>0,\ell=0,1,...,p+1$.
Then
\begin{equation}\label{lem-8eq1}
\int_0^{t_n}  |R_t^{\alpha}v(t)-v(t)|^2 \textrm{d}t \lesssim \mathcal{E}(n,N,\sigma,\alpha,r,p) ,
\end{equation}
where $\mathcal{E}(n,N,\sigma,\alpha,r,p)$ is defined by \eqref{En-1}.
\end{lemma}

\subsubsection{Proof of Theorem \ref{thm1}}
Let $\theta=U-R_t^{\alpha} u$ and $\eta=u-R_t^{\alpha}u$.
Write \eqref{DG-1} into the compact form as
\begin{equation}\label{DG-1b}
\int_{0}^{t_n} \left[\langle\partial_t^\alpha U,X \rangle+ \langle\nabla U, \nabla X \rangle\right]\textrm{d}t
=\int_{0}^{t_n} \langle f+\omega_{1-\alpha}(t)u_0,X \rangle\textrm{d}t,\,\, \forall X \in \mathcal{W}.
\end{equation}
From \eqref{fpde}, we have
\begin{equation}\label{DG-1c}
\int_{0}^{t_n} \left[\langle\partial_t^\alpha u,X \rangle+ \langle\nabla u, \nabla X \rangle\right]\textrm{d}t
=\int_{0}^{t_n} \langle f+\omega_{1-\alpha}(t)u_0,X \rangle\textrm{d}t,\,\, \forall X \in \mathcal{W}.
\end{equation}
From \eqref{DG-1b} and  \eqref{DG-1c}, we obtain the following error equation
\begin{equation} \label{TEEE-5}
A_{\alpha}^n(\theta,X) +A_{0}^n(\nabla\theta,\nabla X)
=  A_{0}^n(\nabla\eta,\nabla X),\quad \forall X \in \mathcal{W}.
\end{equation}
\begin{proof}
Taking  $X=\theta(t)$ in \eqref{TEEE-5}  yields
\begin{equation}\label{TEE-5}\begin{aligned}
A_{\alpha}^n(\theta,\theta) +A_{0}^n(\nabla\theta,\nabla \theta)
=  A_{0}^n(\nabla\theta,\nabla \eta).
\end{aligned}\end{equation}
By  Lemma \ref{lem-8}, we have the following estimate
\begin{equation}\label{TEE-5b}\begin{aligned}
|A_{0}^n(\nabla\theta,\nabla \eta)|\le& A_{0}^n(\nabla\theta,\nabla\theta)+ CA_{0}^n(\nabla\eta,\nabla\eta)\\
\le&  A_{0}^n(\nabla\theta,\nabla\theta)+C\mathcal{E}(n,N,\sigma,\alpha,r,p).
\end{aligned}\end{equation}
Combining \eqref{TEE-5} and \eqref{TEE-5b} yields
\begin{equation}\label{TEE-6}\begin{aligned}
A_{\alpha}^n(\theta,\theta)
\lesssim \mathcal{E}(n,N,\sigma,\alpha,r,p).
\end{aligned}\end{equation}
Using Lemma \ref{lem-4} and $\|U-u\|\le\|\theta\|+\|\eta\|$  arrives
at the desired result, which completes  the proof.
\end{proof}

\subsubsection{Proof of Theorem \ref{thm-2}}
We first introduce the Ritz projection $R_h: H_0^1(\Omega)\rightarrow S_h$, which is defined by
\begin{equation}\begin{aligned}\label{F-2}
\langle \nabla (R_hv-v),\nabla X \rangle=0,\ \ \forall X\in S_h,
\end{aligned}\end{equation}
The projection error $R_hv-v$ has the well-known approximation property \cite{BrennerSR08-B},
\begin{equation}\begin{aligned}\label{F-3}
\|R_hv-v\|_{L^2(\Omega)}+h\|R_hv-v\|_{H^1(\Omega)}\leq C  h^{m+1}\|v\|_{H^{m+1}(\Omega)},\ v\in H^{m+1}(\Omega),\ m\ge0.
\end{aligned}\end{equation}
Denote $\theta=U_h-R_t^{\alpha}R_{h} u$, $\zeta=R_hu-u$,
$\eta_u = R_t^{\alpha}u-u$.
 We can obtain the following error equation of the DG FEM
\begin{equation}\begin{aligned}\label{s233-eq-1a}
&\int_{0}^{t_n}[\langle\partial_t^\alpha \theta(t), X \rangle+ \langle \nabla \theta(t), \nabla X   \rangle]\textrm{d}t\\
=&-\int_{0}^{t_n}[\langle \partial_t^\alpha R_t^{\alpha}\zeta(t), X\rangle
+ \langle \nabla \eta_u(t), \nabla X \rangle]\textrm{d}t
,\quad \forall X\in \mathcal{W}_h.
\end{aligned}\end{equation}

\begin{proof}
Taking $X=\theta$ in \eqref{s233-eq-1a} and applying \eqref{F-3} yields
\begin{equation}\begin{aligned}\label{s233-eq-1}
&A_{\alpha}^n( \theta, \theta)+ A_{0}^n( \nabla \theta, \nabla \theta )\\
=&-A_{\alpha}^n(R_t^{\alpha}\zeta, \theta)-A_{0}^n( \nabla \eta_u, \nabla  \theta)\\
\le & \frac{1}{2} \left(A_{\alpha}^n( \theta, \theta) +
A_{0}^n( \nabla \theta, \nabla  \theta)\right)+
C [A_{\alpha}^n(R_t^{\alpha}\zeta,R_t^{\alpha}\zeta)+ A_{0}^n( \nabla \eta_u, \nabla  \eta_u)].
\end{aligned}\end{equation}
By Lemma \ref{lem-8}, \eqref{F-3}, and Lemmas \ref{lemA2}--\ref{lemA1}, we obtain
\begin{equation}\begin{aligned}\label{RR233-eq-1}
A_{\alpha}^n(R_t^{\alpha}\zeta,R_t^{\alpha}\zeta)
=&\int_0^{t_n}\langle\partial_t^{\alpha}R_t^{\alpha}\zeta,R_t^{\alpha}\zeta\rangle\textrm{d}t
=\int_0^{t_n}\langle\partial_t^{\alpha/2}R_t^{\alpha}\zeta,\ _{t}\partial_{t_n}^{\alpha/2}R_t^{\alpha}\zeta\rangle\textrm{d}t\\
\lesssim & \int_0^{t_n}\left\|\partial_t^{\alpha/2}R_t^{\alpha}\zeta\right\|^2\textrm{d}t\lesssim  h^{2m+2},\\
A_{0}^n(\nabla \eta_u, \nabla  \eta_u)=&\int_0^{t_n}\|\nabla  \eta_u\|^2\textrm{d}t\lesssim \mathcal{E}(n,N,\sigma,\alpha,r,p).
\end{aligned}\end{equation}
Combining \eqref{s233-eq-1} and \eqref{RR233-eq-1}, we obtain
\begin{equation}\begin{aligned}\label{s233-eq-4}
&A_{\alpha}^n( \theta, \theta)+ A_{0}^n( \nabla \theta, \nabla \theta )
\lesssim \mathcal{E}(n,N,\sigma,\alpha,r,p)+h^{2m+2}.
\end{aligned}\end{equation}
Applying Lemma \ref{lem-4} and the  triangle inequality yields
the desired result. The proof is completed.
\end{proof}

\section{Fast time-stepping DG method}\label{sec3}
The main numerical difficulty of the DG method in the previous section is to calculate the following integral
\begin{equation}\label{s5:eq-1}
\int_{I_{n}}X(t) \partial_t^\alpha U(t)\textrm{d}t,\quad X(t)\in \mathcal{P}_p(S_h),U(t)\in \mathcal{W}.
\end{equation}
Direct calculation of \eqref{s5:eq-1} is computationally expensive.
McLean \cite{Mclean2020} presented detailed implementation of calculating \eqref{s5:eq-1}.
Next, we focus on  the fast memory-saving algorithm to calculate \eqref{s5:eq-1},
which is employed to develop the fast time-stepping DG method for \eqref{fpde}.

%
%

\subsection{Review of the  fast method for  calculating the fractional  operators}
We first  introduce the basic idea of the existing fast method for calculating $\partial_t^{-\beta}U(t)$.
For $n\ge 1$, divide $\partial_t^{-\beta} U(t_n)$ into two parts as
\begin{equation}\label{s5:eq-3}\begin{aligned}
\partial_t^{-\beta} U(t_n) =& \underbrace{\int_{t_{n-1}}^{t_{n}}\omega_{\beta}(t_n-t) U(t)\textrm{d}t}_{L^{-\beta,n}U}
+\underbrace{\int_{0}^{t_{n-1}}\omega_{\beta}(t_n-t) U(t)\textrm{d}t}_{H^{-\beta,n}U}.
\end{aligned}\end{equation}
If $U(t)$ is a polynomial or piecewise polynomial, then the  local part $L^{-\beta,n}U$ can be exactly
calculated.
For example,  $U$ is a piecewise linear function defined as
\begin{equation}\label{s5:eq-4}
U(t)|_{t\in (t_{n-1},t_n]} =  \frac{t_n-t}{t_n-t_{n-1}} U^{n-1}_{+} +  \frac{t-t_{n-1}}{t_n-t_{n-1}} U^{n}_{-}.
\end{equation}
Then
\begin{equation}\label{s5:eq-LUn}\begin{aligned}
L^{-\beta,n}U =& \int_{t_{n-1}}^{t_{n}}\omega_{\beta}(t_n-t) U(t)\textrm{d}t
=\frac{\beta\tau_n^{\beta}}{\Gamma(2+\beta)} U^{n-1}_{+}+\frac{\tau_n^{\beta}}{\Gamma(2+\beta)}U^{n}_{-}.
\end{aligned}\end{equation}

Next, we focus on developing the fast memory-saving method to calculate the history part $H^{-\beta,n}U$.

The key step of the fast method is to find a suitable kernel approximation  as \cite{HuangLLZG22}
\begin{equation}\label{ker}
\omega_{\beta}(t) = \sum_{j=1}^Q w^{(-\beta)}_je^{-\lambda_j^{(-\beta)}t} +  \widehat{\varepsilon}(t)\omega_{\beta}(t),
\quad t\in [\delta,T],
\end{equation}
where   $|\widehat{\varepsilon}(t)|\lesssim \varepsilon$ for $ t\in [\delta,T]$, $\varepsilon$ is the relative error.
Set $\widehat{\varepsilon}(t)=0$  for $t\in [0,\delta]$.

There are several approaches to yield \eqref{ker}, here we introduce a simple method based on the trapezoidal rule on the real line that yields the explicit expressions of the weights $ w^{(-\beta)}_j$ and nodes $\lambda_j^{(-\beta)}$.
The kernel $\omega_{\beta}(t)$  has the following integral form
\begin{equation}\label{eq:c18}
\omega_{\beta}(t) =\frac{\sin(\beta\pi)}{\pi}\int_{0}^{\infty}\lambda^{-\beta}e^{-t\lambda}\mathrm{d}\lambda
=\frac{\sin(\beta\pi)}{\pi}\int_{-\infty}^{\infty} e^{-te^{x}+(1-\beta)x} \mathrm{d}x,\quad \beta<1.
\end{equation}
Since $e^{-te^{x}+(1-\beta)x}$ exponentially decays to zero as $|x|\to \infty$ for $t>0$, the exponentially
convergent trapezoidal rule  \cite{Trefethen14} can be used to discretize \eqref{eq:c18}, which yields
\begin{equation}\label{eq:c19}
\begin{aligned}
\omega_{\beta}(t) =& \frac{\sin(\beta\pi)}{\pi}h \sum_{j=-\infty}^{\infty}   e^{-te^{jh}+(1-\beta)jh} +O(e^{-1/h}).
\end{aligned}
\end{equation}
Find $N_1$ and $N_2$ such that  $e^{-te^{jh}+(1-\beta)jh}\le \varepsilon t^{\beta-1}$ for $j\le -N_1 $ and $j\ge N_2$.
So that we can truncate  \eqref{eq:c19} to derive the following kernel approximation  \cite{HuangLLZG22}
\begin{equation}\label{eq:c20}
\begin{aligned}
\omega_{\beta}(t) \approx   \frac{\sin(\beta\pi)}{\pi}h \sum_{j=-N_1+1}^{N_2}   e^{-te^{jh}+(1-\beta)jh}
=\sum_{j=1}^Q w^{(-\beta)}_je^{-\lambda_j^{(-\beta)}t}, \quad t\in[\delta,T],
\end{aligned}
\end{equation}
where
$\lambda_j^{(-\beta)}=e^{-te^{(j-N_1)h}}$ and $w^{(-\beta)}_j=\frac{\sin(\beta\pi)}{\pi}h e^{(1-\beta)(j-N_1)h}.$

With \eqref{eq:c20}, we can construct the fast algorithm for calculating $H^{-\beta,n}U$.
Inserting  \eqref{eq:c20} into
$H^{-\beta,n}U=\int_{0}^{t_{n-1}}\omega_{\beta}(t_n-t) U(t)\textrm{d}t$ yields
\begin{equation}\label{s6:eq-3-2}\begin{aligned}
H^{-\beta,n}U=& \int_{0}^{t_{n-1}}\omega_{\beta}(t_n-t) U(t)\textrm{d}t
\approx \sum_{j=1}^Qw_j^{(-\beta)} \int_{0}^{t_{n-1}}e^{-\lambda_j^{(-\beta)}(t_n-t)} U(t)\textrm{d}t\\
=& \sum_{j=1}^Qw_j^{(-\beta)} e^{-\lambda_j^{(-\beta)}\tau_n}
\int_{0}^{t_{n-1}}e^{-\lambda_j^{(-\beta)}(t_{n-1}-t)} U(t)\textrm{d}t\\
=&\sum_{j=1}^Qw_j^{(-\beta)} e^{-\lambda_j^{(-\beta)}\tau_n}Y_{j}(t_{n-1})
={}_FH^{-\beta,n}U,
\end{aligned}\end{equation}
where $Y_{j}(t)=\int_{0}^{t}e^{-\lambda_j^{(-\beta)}(t-s)} U(s)\textrm{d}s$ is the solution of the following ODE
\begin{equation}\label{s6:eq-3-3}\begin{aligned}
\frac{\textrm{d}}{\textrm{d}t} Y_{j}(t) = -\lambda_j^{(-\beta)}Y_{j}(t) + U(t),\qquad Y_{j}(0) = 0,
\end{aligned}\end{equation}
which can be exactly solved by
\begin{equation}\label{s6:eq-3-4}\begin{aligned}
Y_{j}(t_{n-1}) = e^{-\lambda_j^{(-\beta)}\tau_{n-1}} Y_{j}(t_{n-2})
+ \int_{t_{n-2}}^{t_{n-1}}e^{-\lambda_j^{(-\beta)}(t_{n-1}-t)}U(t)\textrm{d}t,\quad Y_{j}(0)=0.
\end{aligned}\end{equation}

With  \eqref{s6:eq-3-2}, we present the fast method for calculating  $\partial_t^{-\beta} U(t)$  in Algorithm \ref{alg-1}.
\begin{algorithm}
\caption{Fast calculation of  $\partial_t^{-\beta} U(t)$ based on \eqref{ker} or \eqref{eq:c20}.}\label{alg-1}
\begin{itemize}
  \item Step 1. Divide $\partial_t^{-\beta} U(t_n)$ into two parts as \eqref{s5:eq-3}.
  \item Step 2. Calculate the local part $L^{-\beta,n}U$ directly.
  \item Step 3. Approximate the history part $H^{-\beta,n}U$ by   \eqref{s6:eq-3-2}, i.e.,
  \begin{equation}\label{eq:A1}
  {}_FH^{-\beta,n}U=\sum_{j=1}^{Q}w_{j}^{(-\beta)}e^{-\lambda_j\tau_n}Y_j(t_{n-1}),\quad {}_FH^{-\beta,0}U=0.
  \end{equation}
\item Step 4. Output  ${}_FD^{-\beta, n}U = L^{-\beta,n}U +{}_FH^{-\beta,n}U$.
\end{itemize}
\end{algorithm}

It has been shown that the  method \eqref{eq:c20} (or Algorithm \ref{alg-1}) works only for $\beta<1$,  performs badly when $\beta \to 1$, and cannot work for $\beta\ge 1$  \cite{HuangLLZG22}.
This drawback can be tackled by using the following property
\begin{equation}\label{ker2}
\omega_{\beta}(t)=\Big(\prod_{\ell=1}^q(\beta-\ell)\Big)^{-1} {t^q} \omega_{\beta-q}(t),\quad \beta \in \mathbb{R},
\end{equation}
which extends the approach \eqref{eq:c20} to  deal with the case of $\beta\ge 1$.
Assume that \eqref{eq:c20} works well for $\beta\le \beta_0<1$.
Find a smallest integer $q\ge 0$ such that   $\beta-q\le \beta_0$. Applying
\eqref{eq:c20} to $\omega^{(\beta-q)}(t)$ yields
\begin{equation}\label{ker3}
\omega_{\beta}(t)\approx \Big(\prod_{\ell=1}^q(\beta-\ell)\Big)^{-1} \sum_{j=1}^Q  {w_j^{(-\beta+q)}}t^qe^{-{\lambda}_j^{(-\beta+q)}t}
=: \sum_{j=1}^Q  \widehat{w}_j^{(-\beta)}t^qe^{-\widehat{\lambda}_j^{(-\beta)}t}.
\end{equation}
Inserting \eqref{ker3} into $H^{-\beta,n}U$ yields
\begin{equation}\label{s6:eq-3-2b}\begin{aligned}
H^{-\beta,n}U
\approx& \sum_{j=1}^Q\widehat{w}_j^{(-\beta)} \int_{0}^{t_{n-1}}(t_n-t)^qe^{-\widehat{\lambda}_j^{(-\beta)}(t_n-t)} U(t)\textrm{d}t\\
=&\sum_{j=1}^Q\widehat{w}_j^{(-\beta)} e^{-\widehat{\lambda}_j^{(-\beta)}\tau_n}
\sum_{k=0}^q{q\choose k}t_n^{q-k} (-1)^kY_j^{(k)}(t_{n-1})
={}_FH^{-\beta,n}U,
\end{aligned}\end{equation}
where  $Y_j^{(k)}(t_{n-1})$ satisfies
\begin{equation}\label{sec4-4-3}
Y_j^{(k)}(t_{n-1})=e^{\widehat{\lambda}_j^{(-\beta)}\tau_{n-1}}Y_j(t_{n-2})
+ \int_{t_{n-2}}^{t_{n-1}} e^{\widehat{\lambda}_j^{(-\beta)}(t_{n-1}-t)}t^kU(t)
\mathrm{d}t,\quad Y^{(k)}_j(0)=0.
\end{equation}

With \eqref{s6:eq-3-2b}, we present the improved fast Algorithm   \ref{alg-2}.
\begin{algorithm}
\caption{Fast calculation of  $\partial_t^{-\beta} U(t)$ based on \eqref{ker3}.}\label{alg-2}
Replace \eqref{eq:A1} in Algorithm   \ref{alg-1} by  \eqref{s6:eq-3-2b}
to yield the  output ${}_FH^{-\beta,n}U$.
\end{algorithm}

We refer interesting readers to \cite{HuangLLZG22}  for more kernel approximation as \eqref{ker},  where
different kernel approximations are compared, the advantages and pitfalls are displayed, and the corresponding fast
algorithms are developed and numerically verified.

Next, we simply display that Algorithms \ref{alg-1} and \ref{alg-2}
may be used to calculate \eqref{s5:eq-1}. For the linear DG method, we can take
$X\in\{t_n-t,t-t_{n-1}\}$.
Direct calculation shows that
\begin{equation}\label{s3:eq319}
\int_{t_{n-1}}^{t_n}(t_n-t)  \partial_t^{\alpha} U(t)\textrm{d}t=\partial_t^{\alpha-2}U(t_n)
- \partial_t^{\alpha-2} U(t_{n-1})-\tau_n\partial_t^{\alpha-1} U(t_{n-1}).
\end{equation}
Hence, the integral  $\int_{t_{n-1}}^{t_n}(t_n-t)  \partial_t^{\alpha} U(t)\textrm{d}t$
can be approximated by
\begin{equation}\label{s3:eq320}
\int_{t_{n-1}}^{t_n}(t_n-t)  \partial_t^{\alpha} U(t)\textrm{d}t
\approx {}_FD^{\alpha-2,n}U - {}_FD^{\alpha-2,n-1}U-\tau_n{}_FD^{\alpha-1,n-1}U,
\end{equation}
where ${}_FD^{\alpha-k,n}(k=1,2)$ can be calculated by Algorithm  \ref{alg-2}. We can also apply
Algorithm  \ref{alg-1} to obtain ${}_FD^{\alpha-1,n}$ when $\alpha$ is not close to zero.
We can similarly derive
\begin{equation}\label{s3:eq321}
\int_{t_{n-1}}^{t_n}(t-t_{n-1})\partial_t^{\alpha} U(t)\textrm{d}t \approx -{}_FD^{\alpha-2,n}U
+ {}_FD^{\alpha-2,n-1}U+\tau_n{}_FD^{\alpha-1,n}U.
\end{equation}

Numerical simulations show that  \eqref{s3:eq320} is sensitive to roundoff errors.
Therefore, we will not take  \eqref{s3:eq320} and \eqref{s3:eq321} to develop the fast time-stepping DG method in this paper.
In the next section, we  will  develop a new fast algorithm to calculate \eqref{s5:eq-1},
which works well for all $\alpha>0$ and is not sensitive to roundoff errors.

\subsection{Semi-discrete fast time-stepping DG method}
For  $t\in [t_{n-1},t_n],n\ge2$, using the following relation
$$\partial_t^{\alpha} U(t) = {}_{t_{n-2}}\partial_t^{\alpha} U(t)
+ \int_{0}^{t_{n-2}} \omega_{-\alpha}(t-s)U(s)\textrm{d}s,$$
we can obtain
\begin{equation}\label{sec42-eq-2} \begin{aligned}
D^{\alpha,n}(U,X)=\int_{t_{n-1}}^{t_n}\langle X(t),\partial_t^{\alpha} U(t)\rangle \textrm{d}t
=L^{\alpha,n} (U,X) +H^{\alpha,n}(U,X),
\end{aligned}\end{equation}
where
\begin{eqnarray}
L^{\alpha,n}(U,X)&=&\int_{t_{n-1}}^{t_n}\langle{}_{t_{n-2}}\partial_t^{\alpha}  U(t),X(t)\rangle\textrm{d}t,\label{LUX}\\
H^{\alpha,n}(U,X)
&=&\int_{t_{n-1}}^{t_n}  \int_{0}^{t_{n-2}} \omega_{-\alpha}(t-s)\langle U(s),X(t) \rangle\textrm{d}s  \textrm{d}t.\label{HUX}
\end{eqnarray}
The local part $L^{\alpha,n}(U,X)$ can be calculated exactly when $U(t)$ is a piecewise polynomial.

Next, we  develop the fast method to calculate
the history part $H^{\alpha,n}(U,X)$.

%
Similar to \eqref{s6:eq-3-2}, we let $\beta=-\alpha$ in  \eqref{ker} and insert it  into $H^{\alpha,n}(U,X)$, we obtain
\begin{equation} \label{FHUX}\begin{aligned}
H^{\alpha,n}(U,X) \approx &
 \int_{t_{n-1}}^{t_n}\int_{0}^{t_{n-2}}
\sum_{j=1}^Q w_j^{(\alpha)}e^{-\lambda_j(t-s)} \langle U(s),X(t) \rangle\textrm{d}s  \textrm{d}t\\
=& \sum_{j=1}^Q w_j^{(\alpha)}  \int_{t_{n-1}}^{t_n}e^{-\lambda_j(t-t_{n-2})}\bigg\langle\underbrace{\int_{0}^{t_{n-2}}
e^{-\lambda_j(t_{n-2}-s)}   U(s)\textrm{d}s}_{Y_{j}(t_{n-2})},X(t) \bigg\rangle\textrm{d}t\\
=& \sum_{j=1}^Q w_j^{(\alpha)} e^{-\lambda_j \tau_{n-1}}
\bigg\langle Y_{j}(t_{n-2}),\underbrace{\int_{t_{n-1}}^{t_n}X(t) e^{-\lambda_j(t-t_{n-1})}\textrm{d}t}_{\psi^n(X)}\bigg\rangle \\
=&\sum_{j=1}^Q w_j^{(\alpha)}e^{-\lambda_j \tau_{n-1}}\langle \psi^n(X), Y_{j}(t_{n-2})\rangle
={}_{F}H^{\alpha,n}(U,X),
\end{aligned}\end{equation}
where $Y_{j}(t)$ is the solution of the   ODE \eqref{s6:eq-3-3} that can be exactly solved by \eqref{s6:eq-3-4}
and $\psi^n(X)$  is given by
\begin{equation} \label{eq:psi-n}\psi^n(X)=(\tau_n)^2 \left\{\begin{aligned}
& \varphi_1(-\lambda_j\tau_n),&&X=t-t_{n-1},\\
& \varphi_0(-\lambda_j\tau_n)-\varphi_1(-\lambda_j\tau_n),&&X=t_n-t,
\end{aligned}\right.\end{equation}
with $\varphi_k(z)=  \int_{0}^{1}t^{k}e^{zt}\textrm{d}t$.

\begin{remark}
For quadratic DG, we can choose $X\in \{(t-t_{n-1})^2,(t-t_{n-1})(t_n-t),(t_n-t)^2\}$. Direct calculation shows that
\begin{equation*} \label{eq:psi-nb}\psi^n(X)=(\tau_n)^3 \left\{\begin{aligned}
& \varphi_2(-\lambda_j\tau_n),&&X=(t-t_{n-1})^2,\\
& \varphi_1(-\lambda_j\tau_n)-\varphi_2(-\lambda_j\tau_n),&&X=(t-t_{n-1})(t_n-t),\\
& \varphi_0(-\lambda_j\tau_n) - 2\varphi_1(-\lambda_j\tau_n)+\varphi_2(-\lambda_j\tau_n),&&X=(t_n-t)^2.
\end{aligned}\right.\end{equation*}
For $k=0,1,2$, we have
$\varphi_0(z)= z^{-1}(e^z-1)$, $\varphi_1(z)= z^{-2}((z-1)e^z+1)=e^z \frac{e^{-z} - [1 + (-z)]}{(-z)^2}$, and
$\varphi_2(z)= z^{-3}((z^2-2z+2)e^z-2)= 2e^z \frac{e^{-z} - [1 + (-z) + (-z)^2/2]}{(-z)^3}$.
\end{remark}

We present the fast method  for calculating  \eqref{s5:eq-1}  in Algorithm \ref{alg-3}.
\begin{algorithm}
\caption{Fast calculation of  \eqref{s5:eq-1} based on \eqref{ker} or \eqref{eq:c20}.}\label{alg-3}
\begin{itemize}
  \item Step 1. Divide $\int_{t_{n-1}}^{t_n}\langle X(t),\partial_t^{\alpha} U(t)\rangle \textrm{d}t$ into two parts as \eqref{sec42-eq-2}.
  \item Step 2. Calculate the local part $L^{\alpha,n}(U,X)$ directly.
  \item Step 3. Approximate the history part $L^{\alpha,n}(U,X)$ by   \eqref{FHUX}, i.e.,
  \begin{equation}\label{eq:A3}
{}_{F}H^{\alpha,n}(U,X)=\sum_{j=1}^Q w_j^{(\alpha)}e^{-\lambda_j \tau_{n-1}}\langle \psi^n(X), Y_{j}(t_{n-2})\rangle,\quad {}_{F}H^{\alpha,0}(U,X)=0.
  \end{equation}
\item Step 4. Output  ${}_{F}D^{\alpha,n}(U,X) = L^{\alpha,n}(U,X) +{}_{F}H^{\alpha,n}(U,X)$.
\end{itemize}
\end{algorithm}

From \eqref{sec42-eq-2} and \eqref{FHUX}, we can present the semi-discrete fast time-stepping DG method for \eqref{fpde} as:
Given ${}_FU(t)$ for $ t \in (0,t_{n-1}]$,
the discrete solution ${}_FU\in\mathcal{P}_p(\Omega)$ on the next time subinterval $I_n$ is determined by requesting that
\begin{equation}\begin{aligned}\label{FDG-1}
{}_{F}D^{\alpha,n}({}_FU,X)
+\int_{I_n} \langle\nabla {}_FU(t), \nabla X \rangle\textrm{d}t
=\int_{I_n} \langle f(t)+\omega_{1-\alpha}(t)u_0,X \rangle\textrm{d}t,\quad \forall X \in \mathcal{P}_p(\Omega),
\end{aligned}\end{equation}
where ${}_FU = U$ for $t\in (0,t_2]$, $U$ is the solution of \eqref{DG-1}, and ${}_{F}D^{\alpha,n}(\cdot,\cdot)$
is defined by Algorithm \ref{alg-3}.

We have the following error bound.
\begin{theorem} \label{thm3}
Let $U$ and ${}_FU$ be the solutions of \eqref{DG-1} and  \eqref{FDG-1}, respectively.
If $\varepsilon \le c N^{-r\alpha}$ and $c>0$ is sufficiently small, then
$$\int_0^{t_n}\|U(t)-{}_FU(t)\|^2\textrm{d}t\lesssim (\varepsilon t_n^{\alpha} t_1^{-\alpha})^2
\lesssim (\varepsilon n^{r\alpha})^2.$$
\end{theorem}

\subsection{Fully discrete fast time-stepping DG method}
From \eqref{FDG-1}, we can present the fully discrete fast time-stepping DG method for \eqref{fpde} as:
Given ${}_FU_h(t)$ for $ t \in [0,t_{n-1}],n\ge2$,
the discrete solution ${}_FU_h\in \mathcal{P}_{p}(S_{h})$ on the  time subinterval $I_n$ is determined by requesting that
\begin{equation}\begin{aligned}\label{FDG-2}
{}_{F}D^{\alpha,n}({}_FU_h,X)+\int_{I_n} \langle\nabla {}_FU_h(t), \nabla X \rangle\textrm{d}t
= \int_{I_n} \langle f+\omega_{1-\alpha}(t)u_0,X \rangle\textrm{d}t,\quad \forall X \in \mathcal{W}_h,
\end{aligned}\end{equation}
where ${}_FU_h = U_h$ for $t\in [0,t_2]$, $U_h$ is the solution of \eqref{F-1}, and ${}_{F}D^{\alpha,n}(\cdot,\cdot)$
is defined by Algorithm \ref{alg-3}.

We have the following error bound.
\begin{theorem}\label{thm4}
Let $U_h$ and ${}_FU_h$ be the solutions of \eqref{F-1} and  \eqref{FDG-2}, respectively.
If $\varepsilon \le c N^{-r\alpha}$ and $c>0$ is sufficiently small, then
$$\int_0^{t_n}\|U_h(t)-{}_FU_h(t)\|^2\textrm{d}t\lesssim  (\varepsilon t_n^{\alpha} t_1^{-\alpha})^2
\lesssim (\varepsilon n^{r\alpha})^2.$$
\end{theorem}
\subsection{Proofs of Theorems \ref{thm3} and \ref{thm4}}
Proof of Theorem \ref{thm3}.
\begin{proof}
Let $e=U-{}_FU$ and $0<\delta_k\le \tau_{k-1}$. Denote $\widehat{\varepsilon}_k(t)=\widehat{\varepsilon}(t)$ for $t> \delta_k$ and
$\widehat{\varepsilon}_k(t)=0$  for $t\in [0,\delta_k]$. Let $\delta=\min_{1\le k \le N}\delta_k$.
We have $e=0$ for $t\in [0,t_2]$.

From  \eqref{DG-1}, \eqref{ker}, \eqref{sec42-eq-2}, \eqref{FHUX},  and \eqref{FDG-1}, we have the following error equation
\begin{equation}\begin{aligned}\label{FDG-err-1}
&\int_{I_k}\langle {}_{t_{k-2}}\partial_t^{\alpha} e(t),X \rangle\textrm{d}t
+\int_{I_k}\int_0^{t_{k-2}}\omega_{-\alpha}(t-s)\langle U(s),X\rangle\textrm{d}s\textrm{d}t\\
&\qquad\qquad\qquad-{}_{F}H^{\alpha,k}( {}_FU,X)
+\int_{I_k} \langle\nabla e(t), \nabla X \rangle\textrm{d}t
=0,\quad \forall X \in \mathcal{W}_h,
\end{aligned}\end{equation}
By \eqref{ker} and \eqref{FHUX}, we have
\begin{equation}\begin{aligned}\label{FDG-err-2}
{}_{F}H^{\alpha,k}({}_FU,X)
=&\int_{t_{k-1}}^{t_k} \int_{0}^{t_{k-2}} \left[1-\widehat{\varepsilon}_k(t-s)\right]
\omega_{-\alpha}(t-s)\langle{}_F U(s),X \rangle\textrm{d}s\textrm{d}t.
\end{aligned}\end{equation}
Inserting \eqref{FDG-err-2} into \eqref{FDG-err-1} yields
\begin{equation}\begin{aligned}\label{FDG-err-3}
&\int_{I_k}\langle  \partial_t^{\alpha} e(t),X \rangle\textrm{d}t+\int_{I_k} \langle\nabla e(t), \nabla X \rangle\textrm{d}t\\
=&\underbrace{\int_{I_k}\int_0^{t_{k-2}}\widehat{\varepsilon}_k(t-s)\omega_{-\alpha}(t-s)\langle e(s),X\rangle\textrm{d}s\textrm{d}t}_{\mathcal{E}_1^k}\\
&-\underbrace{\int_{I_k}\int_0^{t_{k-2}}\widehat{\varepsilon}_k(t-s)\omega_{-\alpha}(t-s)\langle U(s),X\rangle\textrm{d}s\textrm{d}t}_{\mathcal{E}_2^k}.
\end{aligned}\end{equation}
Taking $X=e(t)$ in \eqref{FDG-err-3} and summing up $k$ from 2 to $n$ yields
\begin{equation}\begin{aligned}\label{FDG-err-4}
A_{\alpha}^n(e,e) \le |A_{\alpha}^n(e,e) + A_{0}^n(\nabla e,\nabla e)|
=|\mathcal{E}_1-\mathcal{E}_2|\le |\mathcal{E}_1|+|\mathcal{E}_2|,
\end{aligned}\end{equation}
where $e(t)=0$ for $t\in(0,t_2]$ is used,  $\mathcal{E}_{j}=\sum_{k=2}^n\mathcal{E}_{j}^k,j=1,2$.

We first estimate $\mathcal{E}_2$. By \eqref{ker}, $\omega_{-\alpha}(t)<0$ for $t>0$,
and $t-\delta\ge t_{k-1}-\delta=\tau_{k-1}-\delta+t_{k-2} \ge t_{k-2}$ for $t\in[t_{k-1},t_k],k\ge 2$, we have
\begin{equation}\begin{aligned}\label{FDG-err-6}
|\mathcal{E}_2|\le&  \sum_{k=2}^n\int_{t_{k-1}}^{t_k}\int_{0}^{t_{k-2}}
|\widehat{\varepsilon}_k(t-s)\omega_{-\alpha}(t-s)| \|U(s)\|\|e(t)\|\textrm{d}s\textrm{d}t\\
\le &-C\varepsilon \sum_{k=2}^n\int_{t_{k-1}}^{t_k}\|e(t)\|\int_{0}^{t_{k-2}}\omega_{-\alpha}(t-s)\|U(s)\|\textrm{d}s\textrm{d}t\\
\le &-C\varepsilon \int_{t_2}^{t_n}\|e(t)\|\int_{0}^{t-\delta}\omega_{-\alpha}(t-s)\|U(s)\|\textrm{d}s\textrm{d}t\\
\le& C\varepsilon\sqrt{\int_{t_2}^{t_n}\|e(t)\|^2\textrm{d}t}
\underbrace{\sqrt{\int_{t_2}^{t_n} \bigg(\int_{0}^{t-\delta}\omega_{-\alpha}(t-s)\|U(s)\|\textrm{d}s\bigg)^2\textrm{d}t}}_{\mathcal{J}[U]},
\end{aligned}\end{equation}
where   $C$ is a positive constant independent of any integer number $n$.

Young's convolution inequality \cite[p. 34]{Adams2003} states that
\begin{equation} \label{Young-conv-eq}
\bigg(\int_a^{b}|(w*g)(t)|^2\textrm{d}t\bigg)^{1/2}
\le \bigg(\int_a^{b} |w(t)|\textrm{d}t\bigg) \bigg(\int_a^{b}|g(t)|^2\textrm{d}t\bigg)^{1/2} ,\quad a\le b,
\end{equation}
if all the integrals in \eqref{Young-conv-eq}.

Choosing $g(t)=\|U(t)\|, w(t) = \omega_{-\alpha}(t+\delta)$, $a=0,b=t_n-\delta$  in \eqref{Young-conv-eq} yields
\begin{equation}\begin{aligned}\label{FDG-err-7}
(\mathcal{J}[U])^2\le&\int_{\delta}^{t_n} \left(\int_{0}^{t-\delta}\omega_{-\alpha}(t-s)\|U(s)\|\textrm{d}s \right)^2\textrm{d}t\\
\le& \left(\int_{0}^{t_n-\delta}   |\omega_{-\alpha}(t+\delta)|\textrm{d}t\right)^2 \int_{0}^{t_n-\delta}  \|U(t)\|^2\textrm{d}t\\
=& \left(\frac{1}{\Gamma(1-\alpha)} \left(\delta^{-\alpha} - t_n^{-\alpha}\right)\right)^2
\int_{0}^{t_n-\delta}  \|U(t)\|^2\textrm{d}t\\
\le &\left(\frac{\delta^{-\alpha}}{\Gamma(1-\alpha)}\right)^2 \int_{0}^{t_n}  \|U(t)\|^2\textrm{d}t.
\end{aligned}\end{equation}
Combining \eqref{FDG-err-6} and \eqref{FDG-err-7}, and using the Cauchy--Schwartz inequality,  we have
\begin{equation}\begin{aligned}\label{FDG-err-8}
|\mathcal{E}_2|\le&\frac{c_0}{\Gamma(1-\alpha)}\int_{t_2}^{t_n}\|e(t)\|^2\textrm{d}t
+ \frac{\Gamma(1-\alpha)(C\varepsilon)^2}{4c_0}(\mathcal{J}[U])^2\\
\le& \frac{1}{\Gamma(1-\alpha)}\left(c_0\int_{0}^{t_n}\|e(t)\|^2\textrm{d}t
+ \frac{\left({C\varepsilon\delta^{-\alpha}}\right)^2}{4c_0}\int_{0}^{t_n}  \|U(t)\|^2\textrm{d}t\right), \quad c_0>0.
\end{aligned}\end{equation}
We can similarly obtain
\begin{equation}\begin{aligned}\label{FDG-err-10}
|\mathcal{E}_1|\le& \frac{1}{\Gamma(1-\alpha)}\left(c_0\int_{0}^{t_n}\|e(t)\|^2\textrm{d}t
+ \frac{\left({C\varepsilon\delta^{-\alpha}}\right)^2}{4c_0}\int_{0}^{t_n}  \|e(t)\|^2\textrm{d}t\right)\\
\le&  \frac{2c_0}{\Gamma(1-\alpha)}\int_{0}^{t_n}\|e(t)\|^2\textrm{d}t
\end{aligned}\end{equation}
when ${\left({C\varepsilon\delta^{-\alpha}}\right)^2}/({4c_0})\le c_0$, i.e., $\varepsilon \le 2c_0\delta^{\alpha}/C$.
%

Combining \eqref{FDG-err-4}, \eqref{FDG-err-8}, and \eqref{FDG-err-10} yields
\begin{equation} \label{FDG-err-11}
A_{\alpha}^n(e,e)\le 3c_0  \int_{0}^{t_n} \frac{\|e(t)\|^2}{\Gamma(1-\alpha)}\textrm{d}t
+\frac{\left({C\varepsilon\delta^{-\alpha}}\right)^2}{4c_0\Gamma(1-\alpha)}\int_{0}^{t_n} \|U(t)\|^2\textrm{d}t.
\end{equation}
Choosing $c_0=t_n^{-\alpha}/8$ and $\delta=t_1$  in \eqref{FDG-err-11}, and using
$$A_{\alpha}^n(e,e)\ge\frac{1}{2}\frac{t_n^{-\alpha}}{\Gamma(1-\alpha)}\int_{0}^{t_n} \|e(t)\|^2\textrm{d}t
=\frac{4c_0}{\Gamma(1-\alpha)}\int_{0}^{t_n} \|e(t)\|^2\textrm{d}t$$
by Lemma \ref{lem-4}, we derive
\begin{equation*}\label{FDG-err-12}
\int_{0}^{t_n}  \|e(t)\|^2\textrm{d}t \lesssim  (\varepsilon t_n^{\alpha}t_1^{-\alpha})^2\int_{0}^{t_n} \|U(t)\|^2\textrm{d}t
 \lesssim  (\varepsilon n^{r\alpha})^2\int_{0}^{t_n} \|U(t)\|^2\textrm{d}t
\end{equation*}
when $\varepsilon \le 2c_0\delta^{\alpha}C^{-1}=(4C)^{-1}t_1^{\alpha}t_n^{-\alpha}=(4C)^{-1} n^{-r\alpha}$.
The proof is completed.
\end{proof}

The proof of Theorem \ref{thm4} is similar to that of Theorem \ref{thm3}, which is omitted.

\section{Numerical results}\label{sec4}
In this section, we perform numerical experiments to verify the efficiency of the scheme.
\begin{example}\label{eg1}
Consider the following initial value problem
\begin{equation}
\label{fffeg1}
\left\{\begin{aligned}
&_{0}^CD_t^\alpha u(t)+u(t)=f(t),\quad t\in (0,T],\ T>0,\\
&u(0)=1.\\
\end{aligned}\right.
\end{equation}
Assume that the exact solution to \eqref{fffeg1} is $u(t)=1+t^{\alpha}+t^{2\alpha}$, and
$$f(t)=1+\Gamma(\alpha+1)+[1+\Gamma(2\alpha+1)/\Gamma(\alpha+1)]t^{\alpha}+t^{2\alpha}.$$
\end{example}


Taking $T=4$, we apply linear and quadratic DG method in the computation, the average errors
$\left(\sum_{n=1}^{N}\tau_n\|e^n\|^2\right)^{1/2}$
are shown in Tables \ref{tab1}--\ref{tab2}.  We observe $(p+1)$-order convergence
when $r\geq\frac{2p+2-\alpha}{1+2\sigma-\alpha}$, which agrees with the theoretical
result.
Figure \ref{fig1}(a) shows the computational time of the fast method and the direct method when $\alpha=0.8$, and we observe that the fast method really outperforms the direct method in efficiency and saves computational cost. The difference between the fast solution and the direct solution is presented in Figure \ref{fig1}(b). We can see that the two solutions are very close, which means the error caused from the trapezoidal rule in fast method is very small.

\begin{table}[!h]
\setlength\tabcolsep{3pt}
\footnotesize
		\caption{The average error   for Example \ref{eg1} when $p=1$.}
		\label{tab1}
		\centering
		\begin{tabular}{|c|c|cc|cc|cc|cc|cc|}
			\hline
			\multirow{2}{*} {$\alpha$} &\multirow{2}{*} {$N$} & \multicolumn{2}{c|}{$r=1$} &\multicolumn{2}{c|}{$r=1.2$} &
\multicolumn{2}{c|}{$r=1.6$} &  \multicolumn{2}{c|}{$r=\frac{4-\alpha}{1+\alpha}$} &  \multicolumn{2}{c|}{$r=3.5$}
\\ \cline{3-12}
&& Error & Rate& Error & Rate &Error & Rate  & Error & Rate & Error & Rate   \\
			\hline
   &32   & 1.67e-02 &-     & 9.66e-03 & -    & 3.26e-03 & -    & 4.21e-04 &-     & 4.41e-04 &-\\
   &64   & 9.67e-03 & 0.79 & 4.97e-03 & 0.96 & 1.32e-03 & 1.30 & 1.18e-04 & 1.84 & 1.24e-04 & 1.83\\
0.2&128  & 5.56e-03 & 0.80 & 2.53e-03 & 0.97 & 5.28e-04 & 1.32 & 3.24e-05 & 1.86 & 3.42e-05 & 1.86\\
   &256  & 3.18e-03 & 0.81 & 1.28e-03 & 0.98 & 2.08e-04 & 1.34 & 8.80e-06 & 1.88 & 9.29e-06 & 1.88\\
   &512  & 1.81e-03 & 0.81 & 6.45e-04 & 0.99 & 8.15e-05 & 1.35 & 2.36e-06 & 1.90 & 2.50e-06 & 1.89\\
\hline
   &32   & 7.32e-03 & -    & 3.24e-03 &-     & 9.12e-04 & -    & 5.80e-04 & -    & 8.37e-04 &-    \\
   &64   & 3.03e-03 & 1.27 & 1.11e-03 & 1.54 & 2.46e-04 & 1.89 & 1.52e-04 & 1.93 & 2.21e-04 & 1.92\\
0.5&128  & 1.23e-03 & 1.30 & 3.74e-04 & 1.57 & 6.52e-05 & 1.92 & 3.93e-05 & 1.95 & 5.75e-05 & 1.94\\
   &256  & 4.88e-04 & 1.33 & 1.23e-04 & 1.60 & 1.70e-05 & 1.94 & 1.01e-05 & 1.96 & 1.48e-05 & 1.96\\
   &512  & 1.91e-04 & 1.35 & 3.99e-05 & 1.63 & 4.39e-06 & 1.95 & 2.56e-06 & 1.97 & 3.79e-06 & 1.97 \\
   \hline
   &32   & 1.28e-03 & -    & 5.73e-04 &-     & 6.03e-04 & -    & 7.03e-04 & -    & 2.15e-03 &-\\
   &64   & 4.32e-04 & 1.56 & 1.57e-04 & 1.87 & 1.54e-04 & 1.97 & 1.79e-04 & 1.97 & 5.47e-04 & 1.97\\
0.8&128  & 1.41e-04 & 1.61 & 4.24e-05 & 1.89 & 3.89e-05 & 1.98 & 4.53e-05 & 1.98 & 1.39e-04 & 1.98\\
   &256  & 4.51e-05 & 1.65 & 1.13e-05 & 1.90 & 9.79e-06 & 1.99 & 1.14e-05 & 1.99 & 3.50e-05 & 1.99\\
   &512  & 1.41e-05 & 1.68 & 3.00e-06 & 1.92 & 2.46e-06 & 1.99 & 2.87e-06 & 1.99 & 8.80e-06 & 1.99\\
    \hline
\end{tabular}
	\end{table}

\begin{figure} 
	\centering
	\subfigure[Computational time]{
		\begin{minipage}[t]{0.4\textwidth}
			\centering
			\includegraphics[width=2in]{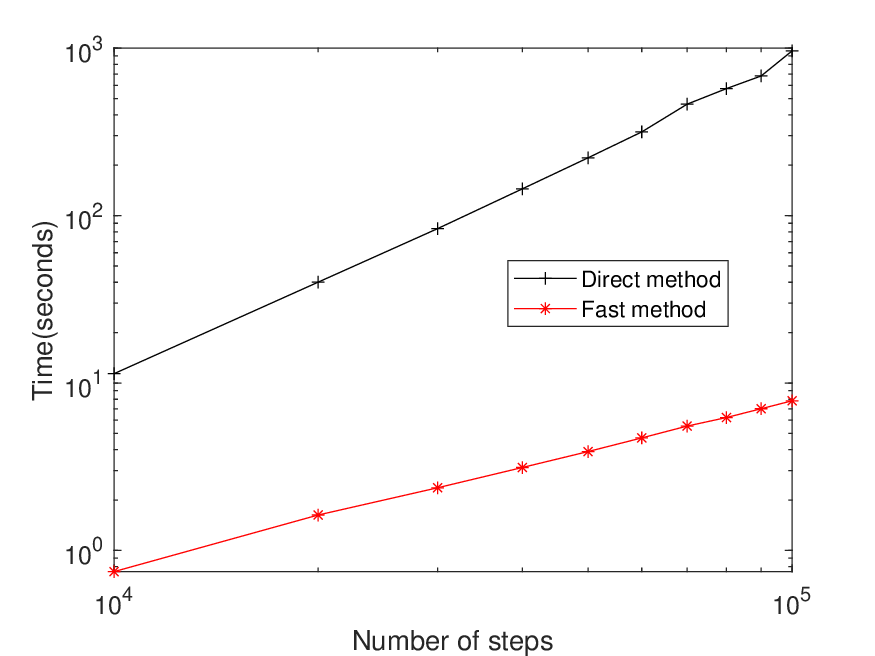}
	\end{minipage}}
	\subfigure[Difference]{
		\begin{minipage}[t]{0.4\textwidth}
			\centering
			\includegraphics[width=2in]{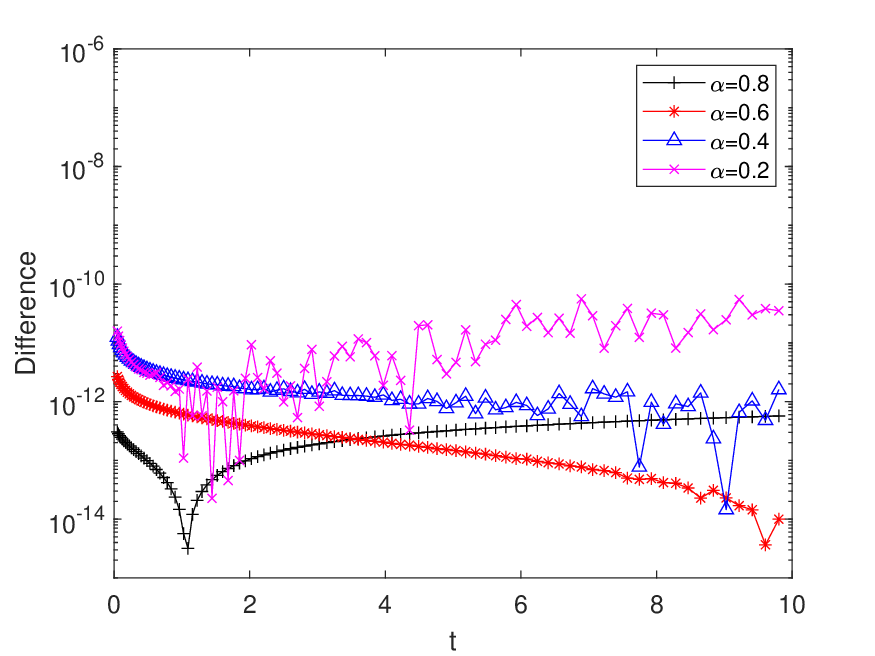}
	\end{minipage}}
	\caption{(a): The computational time of the fast method and direct method; (b):
		The difference between the numerical solutions of the fast method and direct method; Example \ref{eg1}, $Q=256$, $M=10000$, $r=2$, $p=1$.}
	\label{fig1}
\end{figure} 

 \begin{table}[!h]
\setlength\tabcolsep{3pt}
\footnotesize
		\caption{The average error  for Example \ref{eg1} when $p=2$.}
		\label{tab2}
		\centering
		\begin{tabular}{|c|c|cc|cc|cc|cc|cc|}
			\hline
			\multirow{2}{*} {$\alpha$} &\multirow{2}{*} {$N$} & \multicolumn{2}{c|}{$r=2$} &\multicolumn{2}{c|}{$r=2.2$} &
\multicolumn{2}{c|}{$r=2.5$} &  \multicolumn{2}{c|}{$r=\frac{6-\alpha}{1+\alpha}$} &  \multicolumn{2}{c|}{$r=5$}
\\ \cline{3-12}&& Error & Rate& Error & Rate &Error & Rate  & Error & Rate & Error & Rate   \\
			\hline
   &32  & 6.98e-04 &-     & 4.11e-04 & -    & 2.01e-04 &-     & 2.75e-05 & -    & 2.81e-05 &-\\
   &64  & 2.86e-04 & 1.29 & 1.53e-04 & 1.42 & 6.45e-05 & 1.64 & 3.98e-06 & 2.79 & 4.06e-06 & 2.79\\
0.2&128 & 1.15e-04 & 1.31 & 5.61e-05 & 1.45 & 2.03e-05 & 1.67 & 5.54e-07 & 2.84 & 5.63e-07 & 2.85\\
   &256 & 4.62e-05 & 1.32 & 2.03e-05 & 1.47 & 6.32e-06 & 1.69 & 7.54e-08 & 2.88 & 7.65e-08 & 2.88\\
   &512 & 1.83e-05 & 1.34 & 7.29e-06 & 1.48 & 1.94e-06 & 1.70 & 1.01e-08 & 2.90 & 1.02e-08 & 2.90 \\
\hline
   &32  & 2.33e-04 & -    & 1.16e-04 & -    & 4.31e-05 &-     & 5.81e-06 & -    & 6.13e-06 &-\\
   &64  & 7.37e-05 & 1.66 & 3.19e-05 & 1.86 & 9.61e-06 & 2.17 & 8.62e-07 & 2.75 & 1.05e-06 & 2.54\\
0.5&128 & 2.23e-05 & 1.72 & 8.39e-06 & 1.93 & 2.03e-06 & 2.24 & 1.21e-07 & 2.83 & 1.66e-07 & 2.66\\
   &256 & 6.55e-06 & 1.77 & 2.13e-06 & 1.98 & 4.12e-07 & 2.30 & 1.65e-08 & 2.88 & 2.43e-08 & 2.77\\
   &512 & 1.87e-06 & 1.81 & 5.27e-07 & 2.02 & 8.16e-08 & 2.34 & 2.19e-09 & 2.91 & 3.38e-09 & 2.85 \\
\hline
   &32  & 4.04e-05 &-     & 1.72e-05 & -    & 4.86e-06 &-     & 1.76e-06 &-     & 6.77e-06 &-\\
   &64  & 1.11e-05 & 1.87 & 4.09e-06 & 2.08 & 9.25e-07 & 2.39 & 2.39e-07 & 2.88 & 9.03e-07 & 2.91\\
0.8&128 & 2.91e-06 & 1.93 & 9.28e-07 & 2.14 & 1.69e-07 & 2.46 & 3.15e-08 & 2.92 & 1.17e-07 & 2.95\\
   &256 & 7.47e-07 & 1.96 & 2.06e-07 & 2.17 & 3.01e-08 & 2.48 & 4.11e-09 & 2.94 & 1.50e-08 & 2.97\\
   &512 & 1.89e-07 & 1.98 & 4.53e-08 & 2.19 & 5.34e-09 & 2.50 & 5.33e-10 & 2.95 & 1.89e-09 & 2.98 \\
\hline
\end{tabular}
	\end{table}

\begin{table}[!h]
\setlength\tabcolsep{3pt}
\footnotesize
		\caption{The average error   in time for Example \ref{eg22} when $p=1$ and $h=1/256$.}
		\label{tab3}
		\centering
		\begin{tabular}{|c|c|cc|cc|cc|cc|cc|}
			\hline
			\multirow{2}{*} {$\alpha$} &\multirow{2}{*} {$N$} & \multicolumn{2}{c|}{$r=1$} &\multicolumn{2}{c|}{$r=1.2$} &
\multicolumn{2}{c|}{$r=1.6$} &  \multicolumn{2}{c|}{$r=\frac{4-\alpha}{1+\alpha}$} &  \multicolumn{2}{c|}{$r=3.5$}
\\ \cline{3-12}
&& Error & Rate& Error & Rate &Error & Rate  & Error & Rate & Error & Rate   \\
			\hline
   &32  & 1.94e-01 & -    & 1.32e-01 & -    & 6.14e-02 & -    & 6.79e-03 & -    & 6.12e-03 &-\\
   &64  & 1.14e-01 & 0.77 & 7.00e-02 & 0.92 & 2.65e-02 & 1.21 & 1.54e-03 & 2.14 & 1.39e-03 & 2.13\\
0.2&128 & 6.68e-02 & 0.77 & 3.71e-02 & 0.92 & 1.15e-02 & 1.21 & 3.56e-04 & 2.11 & 3.29e-04 & 2.08\\
   &256 & 3.93e-02 & 0.76 & 1.97e-02 & 0.91 & 4.97e-03 & 1.21 & 8.37e-05 & 2.09 & 7.98e-05 & 2.05\\
   &512 & 2.32e-02 & 0.76 & 1.05e-02 & 0.91 & 2.16e-03 & 1.21 & 2.00e-05 & 2.06 & 1.96e-05 & 2.02 \\
\hline
   &32  & 4.69e-02 &-     & 2.46e-02 &-     & 6.83e-03 & -    & 1.78e-03 &  -   & 2.20e-03 &-\\
   &64  & 2.12e-02 & 1.15 & 9.47e-03 & 1.38 & 1.90e-03 & 1.85 & 4.40e-04 & 2.02 & 5.66e-04 & 1.96\\
0.5&128 & 9.50e-03 & 1.16 & 3.57e-03 & 1.41 & 5.06e-04 & 1.91 & 1.08e-04 & 2.02 & 1.44e-04 & 1.98\\
   &256 & 4.20e-03 & 1.18 & 1.30e-03 & 1.45 & 1.30e-04 & 1.96 & 2.69e-05 & 2.01 & 3.62e-05 & 1.99\\
   &512 & 1.82e-03 & 1.21 & 4.61e-04 & 1.50 & 3.25e-05 & 2.00 & 6.68e-06 & 2.01 & 9.09e-06 & 1.99 \\
   \hline
   &32  & 1.17e-02 &-     & 4.74e-03 &-     & 1.22e-03 &-     & 1.15e-03 &-     & 3.16e-03 &-\\
   &64  & 3.76e-03 & 1.64 & 1.15e-03 & 2.04 & 2.60e-04 & 2.23 & 2.81e-04 & 2.03 & 8.44e-04 & 1.90\\
0.8&128 & 1.12e-03 & 1.74 & 2.53e-04 & 2.19 & 6.23e-05 & 2.06 & 7.09e-05 & 1.99 & 2.19e-04 & 1.95\\
   &256 & 3.16e-04 & 1.83 & 5.15e-05 & 2.30 & 1.57e-05 & 1.99 & 1.79e-05 & 1.99 & 5.59e-05 & 1.97\\
   &512 & 8.51e-05 & 1.89 & 1.04e-05 & 2.31 & 3.97e-06 & 1.98 & 4.51e-06 & 1.99 & 1.42e-05 & 1.98 \\
    \hline
\end{tabular}
	\end{table}
 \begin{example}\label{eg22}
Consider the following initial-boundary value problem
\begin{equation}
\label{feg11}
\left\{\begin{aligned}
&_{0}^CD_t^\alpha u(x,t)=\Delta u(x,t)+f(x,t),\quad (x,t)\in \Omega\times (0,T],T>0,\\
&u(x,0)=\sin(x),\quad x \in \bar{\Omega},\\
&u(x,t)=0, \quad(x,t)\in \partial\Omega\times [0,T],
\end{aligned}\right.
\end{equation}
where $\Omega = (0,1)$ and $0< \alpha <1$.
The exact solution to (\ref{eg22}) is $u=(1+t^\alpha+t^{2\alpha})\sin(2\pi x)$ when
$$f(t)=(4\pi^2+\Gamma(\alpha+1)+(4\pi^2+\Gamma(2\alpha+1)/\Gamma(\alpha+1))t^{\alpha}+4\pi^2t^{2\alpha})\sin(2\pi x).$$
\end{example}

\begin{table}[!h]
\setlength\tabcolsep{3pt}
\footnotesize
		\caption{The average error    in space for Example \ref{eg22} when $p=1$ and $N=20000$.}
		\label{tab4}
		\centering
		\begin{tabular}{|c|c|cc|cc|cc|cc|cc|}
			\hline
			\multirow{2}{*} {$\alpha$} &\multirow{2}{*} {$h$} & \multicolumn{2}{c|}{$r=1$} &\multicolumn{2}{c|}{$r=1.2$} &
\multicolumn{2}{c|}{$r=1.6$} &  \multicolumn{2}{c|}{$r=\frac{4-\alpha}{1+\alpha}$} &  \multicolumn{2}{c|}{$r=3.5$}
\\ \cline{3-12}
&& Error & Rate& Error & Rate &Error & Rate  & Error & Rate & Error & Rate   \\
			\hline
   &1/4   & 4.04e-01 &-     & 4.04e-01 & -    & 4.04e-01 &-     & 4.04e-01 &-     & 4.04e-01 &-\\
   &1/8   & 1.24e-01 & 1.70 & 1.24e-01 & 1.70 & 1.24e-01 & 1.70 & 1.24e-01 & 1.70 & 1.24e-01 & 1.70\\
0.2&1/16  & 3.26e-02 & 1.93 & 3.26e-02 & 1.93 & 3.26e-02 & 1.93 & 3.26e-02 & 1.93 & 3.26e-02 & 1.93\\
   &1/32  & 8.35e-03 & 1.97 & 8.24e-03 & 1.98 & 8.23e-03 & 1.98 & 8.23e-03 & 1.98 & 8.23e-03 & 1.98\\
   &1/64  & 2.48e-03 & 1.75 & 2.09e-03 & 1.98 & 2.06e-03 & 2.00 & 2.06e-03 & 2.00 & 2.06e-03 & 2.00\\
\hline
   &1/4   & 3.41e-01 & -    & 3.41e-01 & -    & 3.41e-01 &-     & 3.41e-01 & -    & 3.41e-01 &-\\
   &1/8   & 1.05e-01 & 1.71 & 1.05e-01 & 1.71 & 1.05e-01 & 1.71 & 1.05e-01 & 1.71 & 1.05e-01 & 1.71\\
0.5&1/16  & 2.73e-02 & 1.93 & 2.73e-02 & 1.93 & 2.73e-02 & 1.93 & 2.73e-02 & 1.93 & 2.73e-02 & 1.93\\
   &1/32  & 6.91e-03 & 1.98 & 6.91e-03 & 1.98 & 6.91e-03 & 1.98 & 6.91e-03 & 1.98 & 6.91e-03 & 1.98\\
   &1/64  & 1.73e-03 & 2.00 & 1.73e-03 & 2.00 & 1.73e-03 & 2.00 & 1.73e-03 & 2.00 & 1.73e-03 & 2.00\\
   \hline
   &1/4   & 2.87e-01 & -    & 2.87e-01 &  -   & 2.87e-01 & -    & 2.87e-01 & -    & 2.87e-01 &-\\
   &1/8   & 8.69e-02 & 1.72 & 8.69e-02 & 1.72 & 8.69e-02 & 1.72 & 8.69e-02 & 1.72 & 8.69e-02 & 1.72\\
0.8&1/16  & 2.27e-02 & 1.94 & 2.27e-02 & 1.94 & 2.27e-02 & 1.94 & 2.27e-02 & 1.94 & 2.27e-02 & 1.94\\
   &1/32  & 5.72e-03 & 1.99 & 5.72e-03 & 1.99 & 5.72e-03 & 1.99 & 5.72e-03 & 1.99 & 5.72e-03 & 1.99\\
   &1/64  & 1.43e-03 & 2.00 & 1.43e-03 & 2.00 & 1.43e-03 & 2.00 & 1.43e-03 & 2.00 & 1.43e-03 & 2.00 \\
    \hline
\end{tabular}
	\end{table}

Considering $T=4$, we employ linear DG and linear FEM in the computation.
Table \ref{tab3} gives the average errors and convergence rates in time of Example \ref{eg22} when $\alpha=0.2, 0.5$ and $0.8$, respectively, for various choices of $N$ and $r$. The optimal rates of order $O(N^{-2})$
 is observed  when $r\geq\frac{2p+2-\alpha}{1+2\sigma-\alpha}$.
The errors and convergence rates in space are displayed in Table \ref{tab4},
which second-order accuracy is observed.
The numerical results are consistent with Theorem \ref{thm-2}.

\section{Conclusion}
This work develops a high-order fast time-stepping DG FEM for the subdiffusion equation. We prove the optimal error estimate of the time-stepping DG method, which displays that the error in the $L^2((0,T),L^2(\Omega))$-norm is $O(N^{-p-1} + h^{m+1} + \varepsilon N^{r\alpha})$. Our result improves the error bound in \cite{MusAFN16}.
Numerical examples are given to demonstrate the theoretical analysis.

\appendix
\section{Proof of Lemma \ref{lem-8}} \label{appx-a}
For any function $v\in C([0,T])$, the piecewise Lagrange interpolant $\Pi^p v$    is  defined by
$$ \Pi^p v(t_{n,j})=v(t_{n,j}),\quad t_{n,j}\in [t_{n-1},t_n],\quad 0\le j \le p,\quad 1 \le n \le N,$$
where $\Pi^p v(t)|_{t\in[t_{n-1},t_n]}$ is the   polynomial with degree $p$,
$t_{n,j}=t_{n-1}+j\tau_n/p,0\le j \le p$, $p\ge 1$.
Denote the interpolation error as $\eta_v=v-\Pi^p v$.

Denote the  right-hand-side RL integral operator ${}_{t}\partial_b^{-\alpha}$ as
$${}_{t}\partial_b^{-\alpha}v(t)
=\int_t^b\frac{(s-t)^{\alpha-1}}{\Gamma(\alpha)}v(s)\textrm{d}s,\quad \alpha \ge 0.$$
Then, ${}_{t}\partial_b^{\alpha}$ denotes the  right-hand-side RL derivative operator
of order $\alpha$ in the sense of the principal value \cite{SamKilM93}.

Let  $J_n=(0,t_n)$. Let $J_L^{\alpha}(J_n)$ and $J_R^{\alpha}(J_n)$ be the
left and right factional derivative spaces, respectively.
 Define the semi-norms $|\cdot|_{J_L^{\alpha}(J_n)}$ and $|\cdot|_{J_R^{\alpha}(J_n)}$, respectively,  by
$$|v|_{J_L^{\alpha}(J_n)} = \left(\int_0^{t_n}|\partial_t^{\alpha}v(t)|^2\textrm{d}t\right)^{1/2} \quad \text{and}\quad
|v|_{J_R^{\alpha}(J_n)} = \left(\int_0^{t_n}|{}_t\partial_{t_n}^{\alpha}v(t)|^2\textrm{d}t\right)^{1/2}.$$
The norms $\|\cdot\|_{J_L^{\alpha}(J_n)}$ and $\|\cdot\|_{J_R^{\alpha}(J_n)}$ are, respectively, defined by
$$\|v\|_{J_L^{\alpha}(J_n)} = \left(|v|_{J_L^{\alpha}(J_n)}^2 + \|v\|^2_{L^2(J_n)}\right)^{1/2}   \text{ and }
\|v\|_{J_R^{\alpha}(J_n)} = \left(|v|_{J_R^{\alpha}(J_n)}^2+\|v\|^2_{L^2(J_n)}\right)^{1/2},$$
where
$\|v\|^2_{L^2(J_n)}=\int_0^{t_n} |v(t)|^2\textrm{d}t=|v|_{J_L^{0}(J_n)}.$
Readers can refer to \cite{Roop04} for details.

\begin{lemma}\label{lemA2}
Let $\alpha_1+\alpha_2=\alpha<1$, $\alpha_1,\alpha_2\ge 0$, $w\in J_L^{\alpha}(J_n)$, $v\in J_R^{\alpha_2}(J_n)$. Then
$$\int_0^{t_n}v(t)\partial_t^{\alpha}w(t)\textrm{d}t
=\int_0^{t_n}\partial_t^{\alpha_1}w(t) {}_{t}\partial_{t_n}^{\alpha_2}v(t) \textrm{d}t.$$
\end{lemma}
\begin{proof}
The proof is similar to that of    \cite[Lemma 3.1.4]{Roop04}, the details are omitted.
The proof is complete.
\end{proof}

\begin{lemma}\label{lemA1}
Let  $v\in J_L^{\alpha}(J_n),0<\alpha<1/2$. Then
$$|v|_{J_L^{\alpha}(J_n)}^2\lesssim
\Big|\int_0^{t_n}\partial_t^{\alpha}v(t) {}_{t}\partial_{t_n}^{\alpha}v(t)   \textrm{d}t\Big|
\lesssim |v|_{J_R^{\alpha}(J_n)}^2 \lesssim |v|_{J_L^{\alpha}(J_n)}^2.$$
\end{lemma}
\begin{proof}
Following the proofs of Theorems 3.1.3, 3.1.5, and 3.1.13 in \cite{Roop04}
yields the desired result. The proof is complete.
\end{proof}

\begin{lemma}\label{lemA3}
Let $v\in J_R^{\gamma}(J_n)\cap C(\bar{J}_n)$. Then
$$\|\Pi^pv - v\|_{J_R^{\beta}(J_n)}\lesssim N^{\beta-\gamma} \|v\|_{J_R^{\gamma}(J_n)},\quad 0\le \beta\le \gamma \le 2.$$
\end{lemma}
\begin{proof}
The above error bound follows from
$\|\Pi^pv - v\|_{J_R^{\beta}(J_n)}
\lesssim \|\Pi^pv - v\|_{H^{\beta}(J_n)}\lesssim N^{\beta-\gamma} \|v\|_{H^{\gamma}(J_n)}
\lesssim  N^{\beta-\gamma}  \|v\|_{J_R^{\gamma}(J_n)}$, see
\cite[Theorem 3.1.13]{Roop04} and
\cite[Theorem 3.3.2]{Roop04}  (or \cite[Eq. (4.2.21)]{BrennerSR08-B}).
The proof is complete.
\end{proof}

\begin{lemma}\label{lemA4}
Let $v\in C([0,T])$ and satisfy
\begin{equation} \label{regularity-v}
|v(t)-v(0)|\lesssim t^{\sigma},\quad |v^{(\ell)}(t)| \lesssim 1+ t^{\sigma-\ell} ,  \quad t>0,0< \sigma<1,0\le\ell\le p+1.
\end{equation}
Then
\begin{equation}\label{eqA7}\begin{aligned}
|\eta_v(t)|  \lesssim &  t_{n}^{\sigma-p-1}(t-t_{n})^{p+1},&&\quad t\in[t_{n},t_{n+1}], n\ge 1, \\
|\eta'_v(t)|  \lesssim &  t_{n}^{\sigma-p-1}(t-t_{n})^{p},&&\quad t\in[t_{n},t_{n+1}], n \ge 1\\
|\eta'_v(t)|  \lesssim &  t_1^{\sigma-1} +t^{\sigma-1},&&\quad t\in[0,t_{1}] .
\end{aligned}\end{equation}
\end{lemma}

\begin{proof}
For each $n\ge 1$, let $T_p^n v(t)=\sum_{k=0}^p \frac{v^{(k)}(t_{n})}{k!}(t-t_{n})^k$
be the Taylor polynomial of order $p$ of $v(t)$, satisfying
\begin{equation}\begin{aligned}\label{TD-5a}
v(t)-T_p^n v(t) = \int_{t_{n}}^t\frac{v^{(p+1)}(s)}{p!}(t-s)^p\textrm{d}s.
\end{aligned}\end{equation}
By  $|\Pi^p(T_p^nv - v)|\lesssim |T_p^nv - v|$, the inverse inequality $|(\Pi^p(T_p^nv - v))'|\lesssim \tau_{n+1}^{-1}|T_p^nv - v|$
(see \cite[Eq. (4.5.4)]{BrennerSR08-B})) for $t\in[t_n,t_{n+1}]$,  and the following relation
\begin{equation}\begin{aligned}\label{TD-5b}
\eta_v(t)= v(t)- \Pi^p v(t) =(v(t)-T_p^nv(t)) + \Pi^p(T_p^nv(t) - v(t)),
\end{aligned}\end{equation}
we have
\begin{equation}\begin{aligned}\label{TD-5bb}
|\eta_v(t)|\lesssim& |v(t)-T_p^nv(t)|,\\
|(\eta_v(t))'|\lesssim& |(v(t)-T_p^nv(t))'| + \tau_{n+1}^{-1}|(T_p^nv(t) - v(t))|.
\end{aligned}\end{equation}
For $t\in[t_n,t_{n+1}]$, $n\ge 1$, by \eqref{TD-5bb}, \eqref{TD-5a}, and \eqref{regularity-v},  we obtain
\begin{equation}\label{TD-5}\begin{aligned}
|\eta_v(t)|\lesssim&     \int_{t_{n}}^t|v^{(p+1)}(s)|(t-s)^p\textrm{d}s
\lesssim  t_{n}^{\sigma-p-1}(t-t_{n})^{p+1},\\
|\eta'_v(t)|  \lesssim & \int_{t_{n}}^t|v^{(p+1)}(s)|(t-s)^{p-1}\textrm{d}s
\lesssim  t_{n}^{\sigma-p-1} (t-t_{n})^{p}.
\end{aligned}\end{equation}
For $t\in[0,t_1]$, by
 $\eta_v=v-\Pi^1 v + \Pi^1 v -\Pi^p v $, the inverse inequality
$|(\Pi^1 v -\Pi^p v)'|\lesssim \tau_1^{-1} |\Pi^1 v -\Pi^p v|$,  $(\Pi^1v)'=(v(t_1)-v(0))/\tau_1$,  we derive
\begin{equation} \label{TD-5e}\begin{aligned}
|\eta'_v|  =& |(v-\Pi^1 v)' + (\Pi^1 v -\Pi^p v)'|
\lesssim |(v-\Pi^1 v)'| + \tau_1^{-1}|\Pi^1 v -\Pi^p v|\\
\lesssim &\tau_1^{-1}|v(t_1)-v(0)| + |v'(t)|+ \tau_1^{-1}|v(t)-v(0)|\\
 \lesssim&  t_1^{\sigma-1} +t^{\sigma-1} + t_1^{-1}t^{\sigma}\lesssim t_1^{\sigma-1} +t^{\sigma-1}.
\end{aligned}\end{equation}
The proof is complete.
\end{proof}

\begin{lemma} \label{lem-3}
Let $v$ satisfy \eqref{regularity-v}. Then
\begin{equation} \label{err-L1}
 \int_0^{t_{k-2}} {(t_{k-1}-s)^{-\alpha-1}} |\eta_v(s)|\textrm{d}s \lesssim \Phi(k,N,r,\alpha,\sigma,p), \quad k \ge 2,
\end{equation}
where
\begin{equation}
\Phi(k,N,r,\alpha,\sigma,p)=t_k^{\sigma-\alpha}\left\{\begin{aligned}
& k^{-(p+1-\alpha)},&& r(1+\sigma)\ge p+1,\\
& k^{-r(1+\sigma)},&&  r(1+\sigma)<p+1.
\end{aligned}\right.
\end{equation}
\end{lemma}
\begin{proof}
With \eqref{eqA7},
the error bound \eqref{err-L1} can be similarly proved as that of the  case $p=1$, the details
are omitted, see \cite[Lemma 2.1]{ShenSD18}.
The proof is complete.
\end{proof}

\begin{lemma}\label{lem-5}
Let $v$  satisfy \eqref{regularity-v}.
For $t\in[t_{k-1},t_k],k\ge1$, we have
\begin{equation}\label{sec2-eq-c1}
\int_{t_{k-1}}^{t_k}|\partial_t^{\alpha}\eta_v(t)|^2\textrm{d}t\lesssim \tau_k (\Phi(k,N,r,\alpha,\sigma,p))^2.
\end{equation}
\end{lemma}
\begin{proof}
By $\eta_v(0)=0$ and  \eqref{RL-C},  $\partial_t^{\alpha}\eta_v(t)=\partial_t^{\alpha-1}(\eta_v)'(t)$.
For $t \in [t_{0},t_1]$, by \eqref{TD-5e},
\begin{equation}\label{eq:A-10}\begin{aligned}
|\partial_t^{\alpha}\eta_v(t)|=&\bigg|\int_0^t\frac{(t-s)^{-\alpha}}{\Gamma(1-\alpha)}(\eta_v)'(s)\textrm{d}s\bigg|
\le \int_0^t\frac{(t-s)^{-\alpha}}{\Gamma(1-\alpha)}|(\eta_v)'(s)|\textrm{d}s\\
\lesssim& \int_0^t\frac{(t-s)^{-\alpha}}{\Gamma(1-\alpha)}(t_1^{\sigma-1}+s^{\sigma-1})\textrm{d}s
\lesssim t_1^{\sigma-1}t^{1-\alpha} +t^{\sigma-\alpha}.
\end{aligned}\end{equation}
Hence,
$\int_{0}^{t_1}|\partial_t^{\alpha}\eta_v(t)|^2\textrm{d}t \lesssim t_1^{2\sigma-2\alpha+1}\lesssim
\tau_1(\Phi(1,N,r,\alpha,\sigma,p))^2.$

By the similar reasoning, \eqref{eq:A-10} holds for $t\in[t_1,t_2]$, which leads to $\int_{0}^{t_2}|\partial_t^{\alpha}\eta_v(t)|^2\textrm{d}t\lesssim t_2^{2\sigma-2\alpha+1}\lesssim
\tau_2(\Phi(2,N,r,\alpha,\sigma,p))^2.$

For $t\in[t_{k-1},t_k],k\ge 3$, using \eqref{RL-C} with $\eta_v(t_{k-2})=0$, we have
\begin{equation}\label{sec2-eq-c2}\begin{aligned}
\partial_t^{\alpha}\eta_v(t)=&\underbrace{\int_0^{t_{k-2}}\frac{(t-s)^{-\alpha-1}}{\Gamma(-\alpha)}\eta_v(s)\textrm{d}s}_{J_1}
+ \underbrace{\int_{t_{k-2}}^{t}\frac{(t-s)^{-\alpha}}{\Gamma(1-\alpha)}(\eta_v(s))'\textrm{d}s}_{J_2}.
\end{aligned}\end{equation}
By $(t-s)^{-\alpha-1}\le (t_{k-1}-s)^{-\alpha-1}$ for $(t,s)\in[t_{k-1},t_k]\times [0,t_{k-2}]$ and Lemma \ref{lem-3}, we have
\begin{equation}\label{sec2-eq-c3}\begin{aligned}
|J_1|\le&  \int_0^{t_{k-2}}\frac{(t_{k-1}-s)^{-\alpha-1}}{|\Gamma(-\alpha)|}|\eta_v(s)|\textrm{d}s \lesssim \Phi(k,N,r,\alpha,\sigma,p).
\end{aligned}\end{equation}
By Lemma \ref{lemA4}, we can easily derive
\begin{equation}\label{sec2-eq-c4}\begin{aligned}
|J_2|\le&  \int_{t_{k-2}}^{t}\frac{(t-s)^{-\alpha}}{\Gamma(1-\alpha)}|\eta'_v(s)|\textrm{d}s
\lesssim   (t-t_{k-2})^{1-\alpha}\tau_k^p t_{k-2}^{\sigma-p-1} \\
\lesssim& \tau_k^{p+1-\alpha}t_k^{\sigma-p-1}\lesssim \Phi(k,N,r,\alpha,\sigma,p).
\end{aligned}\end{equation}
Combining \eqref{sec2-eq-c2}--\eqref{sec2-eq-c4} yields
$|\partial_t^{\alpha}\eta_v(t)|\le|J_1|+|J_2|\lesssim \Phi(k,N,r,\alpha,\sigma,p)$,
which leads to  \eqref{sec2-eq-c1} for $k\ge 3$.
The proof   is complete.
\end{proof}

\begin{lemma}\label{lem-6}
Let $t\in[t_{n-1},t_n]$ and $1\leq n \leq N$. If $v$ satisfies   \eqref{regularity-v} and
$\sigma\le p + (p+1)/\alpha$, then
\begin{equation}\label{sec2-eq-c0}
\int_{0}^{t_n}|\partial_t^{\alpha/2}\eta_v(t)|^2\textrm{d}t
\lesssim
\left\{\begin{aligned}
&N^{-r(1+2\sigma-\alpha)},&&1\le r< \frac{2p+2-\alpha}{1+2\sigma-\alpha},\\
&N^{-2p-2+\alpha}(1+\ln(n)),&&r= \frac{2p+2-\alpha}{1+2\sigma-\alpha},\\
&N^{-2p-2+\alpha}t_n^{1+2\sigma-\alpha-\frac{2p+2-\alpha}{r}},&&r> \frac{2p+2-\alpha}{1+2\sigma-\alpha}.
\end{aligned}\right.
\end{equation}

\end{lemma}
\begin{proof}
Using Lemma \ref{lem-5} yields
\begin{equation*}\begin{aligned}
\int_{0}^{t_n}|\partial_t^{\alpha/2}\eta_v(t)|^2\textrm{d}t=& \sum_{k=1}^n \int_{t_{k-1}}^{t_k}|\partial_t^{\alpha/2}\eta_v(t)|^2\textrm{d}t
\lesssim  \sum_{k=1}^n  \tau_k(\Phi(k,N,r,\alpha/2,\sigma,p))^2.
\end{aligned}\end{equation*}

Case $r(1+\sigma)\ge p+1$. By $\tau_k\lesssim k^{r-1}N^{-r}$ and $t_k\lesssim k^rN^{-r}$, we have
\begin{equation}\label{EE2-2-1}\begin{aligned}
\sum_{k=1}^n\tau_k \left(\Phi(k,N,r,\alpha/2,\sigma,p)\right)^2
=& \sum_{k=1}^n k^{-(2p+2-\alpha)} \tau_kt_k^{2\sigma-\alpha}\\
\lesssim & \sum_{k=1}^n k^{-(2p+2-\alpha)} k^{r-1}N^{-r} (k^rN^{-r})^{2\sigma-\alpha}\\
= & N^{-r(1+2\sigma-\alpha)}\sum_{k=1}^n k^{r(1+2\sigma-\alpha)-(2p+3-\alpha)}.
\end{aligned}\end{equation}
Direct calculations show that
\begin{equation}\label{EE2-2-2}
\sum_{k=1}^n k^{r(1+2\sigma-\alpha)-(2p+3-\alpha)} \lesssim \left\{\begin{aligned}
&n^{r(1+2\sigma-\alpha)-(2p+2-\alpha)}&&  \mathrm{if}\ \ r(1+2\sigma-\alpha)>2p+2-\alpha,\\
&\ln(n+1)&&   \mathrm{if}\ \ r(1+2\sigma-\alpha)=2p+2-\alpha,\\
&1&&   \mathrm{if}\ \ r(1+2\sigma-\alpha)<2p+2-\alpha.
\end{aligned}\right.
\end{equation}
Combining \eqref{EE2-2-1} and \eqref{EE2-2-2} yields \eqref{sec2-eq-c0}.

Case $r(1+\sigma)<p+1$. We have
\begin{equation*}\label{EE2-2-4}\begin{aligned}
\sum_{k=1}^n\tau_k \left(\Phi(k,N,r,\alpha/2,\sigma,p)\right)^2
=& \sum_{k=1}^n \tau_kk^{-2r(1+\sigma)} t_k^{2\sigma-\alpha}\\
\lesssim& N^{-r(2\sigma-\alpha)}\sum_{k=1}^n k^{r-1}N^{-r}  k^{-2r(1+\sigma)} k^{r(2\sigma-\alpha)}\\
=& N^{-r(1+2\sigma-\alpha)}\sum_{k=1}^n k^{-1-r(\alpha+1)}
\lesssim N^{-r(1+2\sigma-\alpha)}.
\end{aligned}\end{equation*}
The proof is complete.
\end{proof}

Now, we are in a position to prove Lemma \ref{lem-8}.
\begin{proof}
Denote $P_m(I_k)$ be the polynomial space of degree $m$ on $I_k$ and
$\mathcal{P}_m(J_n)=\{v\in L^2(J_n): v|_{I_k}\in P_m(I_k),\ 1\leq k \leq n\}$.
Let $e(t)=R_t^{\alpha}v(t)-v(t)$, $\theta(t)=R_t^{\alpha} v(t)-\Pi^p  v(t)$, and $\eta_v(t)=\Pi^p v(t)-v(t)$.
Then,
$e(t)=\theta(t) + \eta_v(t)$ and
$$\int_0^{t_n} X(t) \partial_t^{\alpha}\theta(t)\textrm{d}t=-\int_0^{t_n} X(t) \partial_t^{\alpha}\eta_v(t)\textrm{d}t,\qquad X \in \mathcal{P}_m(J_n).$$
Letting $X(t)=\theta(t)$ in the above equation, and using Lemmas \ref{lemA2} and \ref{lemA1}, we obtain
\begin{equation}\label{eq:A14}\begin{aligned}
\int_0^{t_n}|\partial_t^{\alpha/2}\theta(t)|^2\textrm{d}t
\lesssim& \int_0^{t_n} \theta(t) \partial_t^{\alpha}\theta(t)\textrm{d}t
=-\int_0^{t_n}{}_{t}\partial_{t_n}^{\alpha/2} \theta(t) \partial_t^{\alpha/2}\eta_v(t)\textrm{d}t\\
\lesssim& \left(\int_0^{t_n}|{}_{t}\partial_{t_n}^{\alpha/2} \theta(t)|^2 \textrm{d}t\right)^{1/2}
\left(\int_0^{t_n}  |\partial_t^{\alpha/2}\eta_v(t)|^2\textrm{d}t\right)^{1/2}\\
\lesssim& |\theta|_{J_L^{\alpha/2}(J_n)} |\eta_v|_{J_L^{\alpha/2}(J_n)}.
\end{aligned}\end{equation}
The above inequality   yields  $|\theta|_{J_L^{\alpha/2}(J_n)} \lesssim |\eta_v|_{J_L^{\alpha/2}(J_n)}$.
Using the triangular inequality $|e|_{J_L^{\alpha/2}(J_n)}\le |\theta|_{J_L^{\alpha/2}(J_n)} + |\eta_v|_{J_L^{\alpha/2}(J_n)}$
and Lemma \ref{lem-6},
we obtain
\begin{equation}\label{eq:A15}
|e|_{J_L^{\alpha/2}(J_n)}^2  \lesssim
N^{\alpha} \mathcal{E}(n,N,\sigma,\alpha,r,p).
\end{equation}

Introduce $w$ as the solution to the adjoint problem
\begin{equation}\label{eq:A16}
{}_{t}\partial_{t_n}^{\alpha}w(t)= e(t),\qquad w(t_n) =0.
\end{equation}
By Lemma \ref{lemA1} and the above equation, $w$     satisfies
\begin{equation}\label{eq:A17}
\|w\|_{J_L^{\alpha}(J_n)} \lesssim \|w\|_{J_R^{\alpha}(J_n)}  \lesssim   \|e\|_{L^2(J_n)}.
\end{equation}

Multiplying $e(t)$ on both sides of \eqref{eq:A16} and integrating over $J_n$, we obtain
\begin{equation}\label{eq:A18}\begin{aligned}
\|e\|^2_{L^2(J_n)}
=& \int_0^{t_n}e(t) {}_{t}\partial_{t_n}^{\alpha}w(t)     \textrm{d}t
=  \int_0^{t_n} w(t)\partial_{t}^{\alpha}   e(t) \textrm{d}t\qquad \text{(By Lemma \ref{lemA2})}\\
=&  \int_0^{t_n} (w(t)-\Pi^pw(t))\,\partial_{t}^{\alpha} e(t)\textrm{d}t\qquad (\text{By \eqref{Rtalf}})\\
=& -\int_0^{t_n} {}_{t}\partial_{t_n}^{\alpha/2}\eta_w(t)\,\partial_{t}^{\alpha/2} e(t)\textrm{d}t\qquad \text{(By Lemma \ref{lemA2})}\\
\le  & |e|_{J_L^{\alpha/2}(J_n)}|\eta_w|_{J_R^{\alpha/2}(J_n)}  \\
\lesssim & N^{-\alpha/2} |e|_{J_L^{\alpha/2}(J_n)}\|w\|_{J_R^{\alpha}(J_n)} \qquad \text{(By Lemma  \ref{lemA3})}\\
\lesssim& N^{-\alpha/2} |e|_{J_L^{\alpha/2}(J_n)} \|e\|_{L^2(J_n)}. \qquad \text{(By \eqref{eq:A17})}
\end{aligned}\end{equation}
The above inequality and \eqref{eq:A15} yield \eqref{lem-8eq1}.
The proof is complete.
\end{proof}

\bibliographystyle{siam}
\bibliography{fded19m1y17}

\end{document}